\documentclass[11pt,reqno,twoside]{amsart}
\usepackage{amssymb,amsmath,amsthm,soul,color,paralist}
\usepackage{t1enc}
\usepackage[cp1250]{inputenc}
\usepackage{a4,indentfirst,latexsym}
\usepackage{graphics}
\usepackage{mathrsfs}
\usepackage{cite,enumitem,graphicx}
\usepackage[colorlinks=true,urlcolor=blue,
citecolor=red,linkcolor=blue,linktocpage,pdfpagelabels,
bookmarksnumbered,bookmarksopen]{hyperref}
\usepackage[english]{babel}
\usepackage[left=2.50cm,right=2.50cm,top=2.72cm,bottom=2.72cm]{geometry}
\usepackage[metapost]{mfpic}
\usepackage[hyperpageref]{backref}
\usepackage[colorinlistoftodos]{todonotes}
\usepackage[normalem]{ulem}

\makeatletter
\providecommand\@dotsep{5}
\def\listtodoname{List of Todos}
\def\listoftodos{\@starttoc{tdo}\listtodoname}
\makeatother

\numberwithin{equation}{section}

\newcommand{\h}{H^{s}_{\e}}
\newcommand{\R}{\mathbb{R}}
\newcommand{\2}{2^{*}_{s}}

\newcommand{\C}{\mathbb{C}}

\newcommand{\N}{\mathcal{N}}

\DeclareMathOperator{\dive}{div}
\DeclareMathOperator{\supp}{supp}
\DeclareMathOperator{\e}{\varepsilon}

\newtheorem{prop}{Proposition}[section]
\newtheorem{lem}{Lemma}[section]
\newtheorem{thm}{Theorem}[section]

\newtheorem{cor}{Corollary}[section]

\newtheorem{remark}{Remark}[section]


\begin{document}
\title[Fractional Schr\"odinger-Poisson type equation with magnetic field]{Multiplicity and concentration results for a fractional Schr\"odinger-Poisson type equation with magnetic field}

\author[V. Ambrosio]{Vincenzo Ambrosio}
\address{Vincenzo Ambrosio\hfill\break\indent 
Dipartimento di Ingegneria Industriale e Scienze Matematiche,  \hfill\break\indent
Universit\`a Politecnica delle Marche, \hfill\break\indent
via Brecce Bianche 12, \hfill\break\indent
60131 Ancona (Italy)}
%\address{Vincenzo Ambrosio\hfill\break\indent 
%Dipartimento di Scienze Pure e Applicate (DiSPeA),\hfill\break\indent
%Universit\`a degli Studi di Urbino `Carlo Bo'\hfill\break\indent
%Piazza della Repubblica, 13\hfill\break\indent
%61029 Urbino (Pesaro e Urbino, Italy)}
\email{vincenzo.ambrosio2@unina.it}

%\thanks{The authors are partially supported by grants of the group GNAMPA of INdAM}
%	 The first author is also partially supported by FRA project of Politecnico di Bari.


%\date{\today}

\begin{abstract}
This paper is devoted to the study of fractional Schr\"odinger-Poisson type equations with magnetic field of the type
\begin{equation*}
\varepsilon^{2s}(-\Delta)_{A/\varepsilon}^{s}u+V(x)u+\e^{-2t}(|x|^{2t-3}*|u|^{2})u=f(|u|^{2})u \quad \mbox{ in } \mathbb{R}^{3},
\end{equation*}
where $\varepsilon>0$ is a parameter, $s,t\in (0, 1)$ are such that $2s+2t>3$, $A:\mathbb{R}^{3}\rightarrow \mathbb{R}^{3}$ is a smooth magnetic potential, $(-\Delta)^{s}_{A}$ is the fractional magnetic Laplacian, 
$V:\R^{3}\rightarrow \R$ is a continuous electric potential and $f:\mathbb{R}\rightarrow \mathbb{R}$ is a $C^{1}$ subcritical nonlinear term. 
Using variational methods, we obtain the existence, multiplicity and concentration of nontrivial solutions for $\e>0$ small enough.
\end{abstract}

\maketitle

\section{introduction}

\noindent 
This paper deals with the following fractional nonlinear Schr\"odinger-Poisson type equation
\begin{equation}\label{P}
\varepsilon^{2s}(-\Delta)_{A/\varepsilon}^{s}u+V(x)u+\e^{-2t}(|x|^{2t-3}*|u|^{2})u=f(|u|^{2})u \quad \mbox{ in } \mathbb{R}^{3},
\end{equation}
where $\e>0$ is a parameter and $s,t\in (0, 1)$ are such that $2s+2t>3$.
% $V\in C(\R^{3}, \R)$ is an electric potential and $A\in C^{0,\alpha}(\R^{3},\R^{3})$, with $\alpha\in(0,1]$,  is a magnetic potential.  
Throughout the paper, we assume that $V:\R^{3}\rightarrow \R$ is a continuous potential verifying the following condition 
\begin{equation}
\label{RV}
\tag{RV}
V_{\infty}=\liminf_{|x|\rightarrow \infty} V(x)>V_{0}=\inf_{x\in \R^{N}} V(x)>0
\end{equation}
introduced by Rabinowitz in \cite{Rab}. Here we assume that $V_{\infty}\in (0, \infty]$.
The nonlinearity $f: \R\rightarrow \R$ is a $C^{1}$ function such that $f(t)=0$ for $t\leq 0$, and 
\begin{compactenum}[$(f_1)$]
%\item \label{f1}$f(t)=0$ for $t\leq 0$;
\item $\displaystyle{\lim_{t\rightarrow 0} \frac{f(t)}{t}=0}$;
\item there exists $q\in (4, 2^{*}_{s})$, where $2^{*}_{s}=6/(3-2s)$, such that $\lim_{t\rightarrow \infty} f(t)/t^{\frac{q-2}{2}}=0$;
\item there exists $\theta>4$ such that $0<\frac{\theta}{2} F(t)\leq t f(t)$ for any $t>0$, where $F(t)=\int_{0}^{t} f(\tau)d\tau$;
\item $t\mapsto \frac{f(t)}{t}$ is increasing in $(0, \infty)$;
\item there exist $\sigma\in (4, 2^{*}_{s})$ and a constant $C_{\sigma}>0$ such that $f'(t)t-f(t)\geq C_{\sigma} t^{\frac{\sigma-2}{2}}$ for any $t\geq 0$.
\end{compactenum} 
We note that assumption $(f_2)$ forces to be $s\in (3/4,1)$.
The nonlocal operator $(-\Delta)^{s}_{A}$ is the fractional magnetic Laplacian which may be defined along smooth functions $u:\R^3\to\C$ by setting
\begin{equation}\label{operator}
(-\Delta)^{s}_{A}u(x)
:=c_{3,s} \lim_{r\rightarrow 0} \int_{B_{r}^{c}(x)} \frac{u(x)-e^{\imath (x-y)\cdot A(\frac{x+y}{2})} u(y)}{|x-y|^{3+2s}} dy,
\quad
c_{3,s}:=\frac{4^{s}\Gamma\left(\frac{3+2s}{2}\right)}{\pi^{3/2}|\Gamma(-s)|}.
\end{equation}
This operator has been introduced in \cite{DS, I10} and replies essentially on the L\'evy-Khintchine formula for the generator of a general L\'evy process. 
%As stated in \cite{SV}, up to correcting the operator by the factor $(1-s)$, it follows that $(-\Delta)^{s}_{A}u$ converges to $-(\nabla u-\imath A)^{2}u$ as $s\rightarrow 1$, that is Thus, up to normalization, the nonlocal case can be seen as an approximation of the local one.  If the magnetic field A ? 0,
In absence of magnetic field, that is $A\equiv 0$, the operator $(-\Delta)^{s}_{A}$ reduces to the fractional Laplacian operator $(-\Delta)^{s}$ which has been extensively considered in the last years due to its great application in a lot of pure and applied mathematical problems; see \cite{DPV, MBRS} for more details.
As showed in \cite{PSV} and \cite{SV}, up to correcting the operator by the factor $(1-s)$, it is possible to see that, as $s\rightarrow 1$, $(-\Delta)^{s}_{A}u$ converges to the magnetic Laplacian $-(\nabla-\imath A)^{2}u$ defined as
$$
-(\nabla -\imath A)^{2}u=-\Delta u+2\imath A(x)\cdot \nabla u+|A(x)|^{2}u+\imath u \dive(A(x));
$$
see \cite{LL} for more details. 
In fact, the study of our problem \eqref{P} is motivated by some interesting results obtained for the following nonlinear Schr\"odinger equation with magnetic field
$$
-\left(\frac{\e}{\imath}\nabla -A(z)\right)^{2}u+V(x)u=f(x, |u|^{2})u \quad \mbox{ in } \R^{N},
$$
for which several existence and multiplicity results have been established; see \cite{AFF, AFY, AS, Cingolani, CCS, CS, EL, K}.\\
This equation plays a very important role when we look for standing wave solutions $\psi(x,t)=u(x)e^{-\imath\frac{E}{\hbar}t}$, with $E\in \R$, to the following time dependent magnetic Schr\"odinger equation
$$
\imath \hbar \frac{\partial \psi}{\partial t}=\left( \frac{\hbar}{\imath}\nabla-A(z) \right)^{2}\psi+(V(z)+E)\psi-f(|\psi|^{2})\psi \quad \mbox{ in } \R^{N}\times \R,
$$
where $\hbar$ is the Planck's constant.
Then, one is interested in the existence and the shape of such solutions when $\hbar=\e\rightarrow 0$. Indeed, it is well known that the transition from quantum mechanics to classical mechanics can be formally performed by sending the Planck's constant to zero.
%Indeed, it is well known that the transition from quantum mechanics to classical mechanics can be formally performed by taking the limit as $\hbar=\e\rightarrow 0$.\\

When $A\equiv 0$ and $\phi_{|u|}^{t}:=|x|^{2t-3}*|u|^{2}=0$, equation \eqref{P} becomes the fractional Schr\"odinger equation
\begin{equation}\label{FSE}
\e^{2s}(-\Delta)^{s}u+V(x)u=f(u^{2})u \quad \mbox{ in } \R^{3},
\end{equation}
formulated by Laskin \cite{Laskin} as a result of expanding the Feynman path integral, from the Brownian-like to the L\'evy-like quantum mechanical paths. We recall that equation \eqref{FSE} has attracted the attention of many researchers and different results concerning the existence, multiplicity and concentration behavior as $\e\rightarrow 0$ have been established for it; see for instance \cite{AM, A1, A3, DDPW, DMV, Secchi}. 

On the other hand, when $A\equiv 0$ and $\phi_{|u|}^{t}\neq 0$, equation \eqref{P} can be deduced from a fractional Schr\"odinger-Poisson system
of the type 
\begin{equation}\label{FSPS} 
\left\{
\begin{array}{ll}
\e^{2s}(-\Delta)^{s} u+V(x) \phi u = g(x, u) &\mbox{ in } \R^{3} \\
\e^{2t}(-\Delta)^{t} \phi= u^{2} &\mbox{ in } \R^{3}.
\end{array}
\right.
\end{equation}
When $s=1$,  system  \eqref{FSPS} becomes the classical Schr\"odinger-Poisson system which
%which appears in quantum mechanics models (see e.g. [37]) and in semiconductor theory [41]. 
arises in the study of quantum mechanics models \cite{BBL} and in semiconductor theory \cite{Mar}. 
%Such systems have been introduced in \cite{BF} to 
%describe systems of identical charged particles interacting each other in the case that effects of magnetic field could be ignored and its solution represents, in particular, a standing wave for such a system. 
These systems have been widely studied in the last two decades; see \cite{AdAP, BF, ruiz, ZZ} for unperturbed problems (i.e. $\e=1$) and \cite{A, DW, He, HL, WTXZ, Y} for perturbed problems (i.e. $\e>0$ small).
%without magnetic field. 
%When $A\neq 0$, we refer to \cite{ZS} in which the authors proved the multiplicity and concentration of nontrivial solutions for a magnetic Schr\"odinger-Poisson type equation.

In the nonlocal framework, with $A\equiv 0$, we can mention only few results for \eqref{FSPS}. For instance, Giammetta \cite{G} investigated the local and global well-posedness  of a one dimensional fractional Schr\"odinger-Poisson system in which $\e=1$ and the fractional diffusion appears only in the Poisson equation. %in \eqref{FSPS}.
Zhang et al. \cite{ZDS} dealt with the existence of positive solutions to \eqref{FSPS} involving a general nonlinearity having subcritical or critical growth.  
Murcia and Siciliano \cite{MS} proved that, for $\e>0$ small enough, the number of positive solutions is estimated below by the Ljusternick-Schnirelmann category of the set of minima of the potential.
%, for suitably small $\e$, the number of positive solutions to a doubly singularly perturbed fractional Schr\"odinger-Poisson system is estimated below by the Ljusternick-Schnirelmann category of the set of minima of the potential. 
Teng \cite{teng}  studied the existence of ground state solutions for a critical fractional Schr\"odinger-Poisson system like \eqref{FSPS} with $\e=1$.
Liu and Zhang \cite{LZ} focused on the multiplicity and concentration of solutions to \eqref{FSPS} involving the critical exponent and under assumption \eqref{RV}.
%; {\color{red}{see also \cite{Accm}.}}
On the other hand, in recent years, appeared some interesting results for fractional magnetic Schr\"odinger equations of the type
\begin{equation}\label{FMSE}
\e^{2s}(-\Delta)^{s}_{A}u+V(x)u=f(x, |u|^{2})u \quad \mbox{ in } \R^{N}.
\end{equation}
For instance, d'Avenia and Squassina \cite{DS} considered the existence of ground state solutions for an autonomous fractional magnetic problem. Zhang et al. \cite{ZSZ} focused on the study of nontrivial solutions for a critical magnetic Schr\"odinger equation. Mingqi et al. \cite{MPSZ} dealt with the existence and multiplicity for a fractional magnetic Kirchhoff problem with subcritical nonlinearities. Fiscella et al. \cite{FPV} obtained a multiplicity result for a fractional magnetic boundary value problem. In \cite{AD} the author and d'Avenia investigated the existence and multiplicity of nontrivial solutions to \eqref{FMSE} under the condition \eqref{RV}. For other papers concerning the fractional magnetic Laplacian we refer to \cite{Adpde, Amjm, PSV0} and references therein.
%{\color{red}{In \cite{Adpde} the concentration phenomenon for a fractional Choquard equation with magnetic field is studied.}}
After an accurate bibliographic review, we have realised that no results for fractional magnetic Schr\"odinger-Poisson equations are available in literature.
Strongly motivated by this fact and by the papers \cite{AFF, AD, He}, in this work we focus our attention on the existence, multiplicity and concentration of nontrivial solutions to \eqref{P}. In particular way, we are interested in relating the number of nontrivial solutions of \eqref{P} with the set of global minima of $V$ given by
% WerecallthatifY isagivenclosedsetofa topological space X, we denote by catX(Y) the Ljusternik-Schnirelmann category of Y in X, that is the least number of closed and contractible sets in X which cover Y .
%In order to state precisely our main result, we introduce the sets
\begin{equation}\label{defM}
M=\{x\in \R^{3}: V(x)=V_{0}\}.
\end{equation}
For any $\delta>0$, we also define
\begin{equation}
M_{\delta}=\{x\in \R^{3}: \operatorname{dist}(x, M)\leq \delta\}.
\end{equation}
In order to state precisely our main result, we recall that if $Y$ is a given closed set of a topological space $X$, we denote by $cat_{X}(Y)$ the Ljusternik-Schnirelmann category of $Y$ in $X$, that is the least number of closed and contractible sets in $X$ which cover $Y$. 
Then we prove the following main result:
\begin{thm}\label{thm1}
Assume that \eqref{RV} and $(f_1)$-$(f_5)$ hold. Then, for any given $\delta>0$ there exists $\e_{\delta}>0$ such that, for any $\e\in (0, \e_{\delta})$, problem \eqref{P} has at least $cat_{M_{\delta}}(M)$ nontrivial solutions. Moreover, if $u_{\e}$ denotes one of these solutions and $x_{\e}$ be the global maximum point of $|u_{\e}|$, we have 
$$
\lim_{\e\rightarrow 0} V(x_{\e})=V_{0}
$$	
and there exists $\tilde{C}>0$ such that
$$
|u_{\e}(x)|\leq  \frac{\tilde{C} \e^{3+2s}}{\e^{3+2s}+|x-x_{\e}|^{3+2s}} \quad \forall x\in \R^{3}.
$$
\end{thm}
The proof of Theorem \ref{thm1} is obtained applying suitable variational methods.
Firstly, we use the change of variable $x\mapsto \e x$ to see that problem (\ref{P}) is equivalent to the following one
\begin{equation}\label{Pe}
(-\Delta)_{A_{\e}}^{s} u + V_{\e}( x)u +(|x|^{2t-3}*|u|^{2})u=  f(|u|^{2})u  \mbox{ in } \R^{3},
\end{equation}
where $A_{\e}(x)=A(\e x)$ and $V_{\e}(x)=V(\e x)$. Then, we look for weak solutions to \eqref{Pe} studying the critical points of the corresponding Euler-Lagrange functional. The assumption on the behavior of $V$ at infinity and the superlinear-4 growth condition on $f$, will play a fundamental role to deduce some compactness properties; see Proposition \ref{prop2.2}. The H\"older regularity of the magnetic field together with the fractional diamagnetic inequality (Lemma \ref{DI}) and some interesting decay properties of the positive solutions of the limit problem associated with \eqref{Pe} (see proof of Lemma \ref{FS}), will be crucial to obtain the existence of a solution to \eqref{Pe} for small $\e$; see Theorem \ref{Exthm}. We point out that the restriction $2s+2t>3$ will be used to prove that the operator $\Psi(u)=\int_{\R^{3}} \phi^{t}_{|u|} |u|^{2}dx$ and its differential possess a Brezis-Lieb splitting property \cite{BL}; see Lemma \ref{CCS}.
%in order to have a Brezis-Lieb splitting \cite{BL} for the operator $\Psi(u)=\int_{\R^{3}} \phi^{t}_{|u|} |u|^{2}dx$ (see Lemma \ref{CCS}).
After that, we use some appropriate tools like the barycenter map and the Ljusternik-Schnirelmann theory to prove a multiplicity result for \eqref{Pe}.  Finally, we study the concentration of solutions by combining a Moser iteration scheme \cite{Moser} with an approximation argument inspired by the Kato's inequality \cite{Kato} for the magnetic Laplacian; see Lemma \ref{moser}. We also provide a decay estimate for the modulus of solutions to \eqref{P} with the help of papers \cite{AM, FQT}.

The paper is organized as follows: in Section $2$ we recall some results for the fractional magnetic spaces and we give some useful lemmas. In Section $3$ we introduce the functional associated with \eqref{Pe} and we also consider the corresponding  autonomous problem. In Section $4$ we study the compactness properties of the functional and in Section $5$ we give a first existence result. The Section $6$ is dedicated to the multiplicity result for \eqref{P} and in the last section we study the behavior of maximum points of the modulus of nontrivial solutions.

\section{Preliminaries}\label{sec2}
For the reader's convenience, in this section, we fix the notations and we give some lemmas which will be used in the next sections. 
Let us denote by $L^{2}(\R^{3}, \C)$ the set of functions $u:\R^{3}\rightarrow \C$ such that $\int_{\R^{3}}|u|^{2}\, dx<\infty$.
It is clear that $L^{2}(\R^{3}, \C)$ is a Hilbert space endowed with the inner product
$$
(u,v)_{2}=\int_{\R^{3}} u\bar{v}\, dx,
$$
where the bar denotes the complex conjugation.
Let $A\in C(\R^{3}, \R^{3})$ be a continuous magnetic field. Consider the magnetic Gagliardo semi-norm of a function function $u:\R^3\to\C$ by setting
$$
[u]^{2}_{A}:=\frac{c_{3,s}}{2}\iint_{\R^{6}} \frac{|u(x)-e^{\imath (x-y)\cdot A(\frac{x+y}{2})} u(y)|^{2}}{|x-y|^{3+2s}} \, dxdy,
$$
and set
$$
D_A^s(\R^3,\C)
:=
\left\{
u\in L^{2_s^*}(\R^3,\C) : [u]^{2}_{A}<\infty
\right\}.
$$
Let us introduce the Hilbert space
$$
H^{s}_{\e}:=
\left\{
u\in D_{A_{\e}}^s(\R^3,\C): \int_{\R^{3}} V_{\e} (x) |u|^{2}\, dx <\infty
\right\}
$$ 
endowed with the scalar product
\begin{align*}
\langle u , v \rangle_{\e}
&:=
\Re	\int_{\R^{3}} V_{\e} (x) u \bar{v} dx\\
&\qquad
+ \frac{c_{3,s}}{2}\Re\iint_{\R^{6}} \frac{(u(x)-e^{\imath(x-y)\cdot A_{\e}(\frac{x+y}{2})} u(y))\overline{(v(x)-e^{\imath(x-y)\cdot A_{\e}(\frac{x+y}{2})}v(y))}}{|x-y|^{3+2s}} dx dy
\end{align*}
and let
$$
\|u\|_{\e}:=\sqrt{\langle u , u \rangle_{\e}}.
$$
We recall the following useful properties for the space $\h$ (see \cite{AD, DS} for more details):
\begin{lem}\label{density}
The space $\h$ is complete and $C_c^\infty(\R^3,\C)$ is dense in $\h$. 
\end{lem}
%We recall that the pointwise diamagnetic inequality
%\[
%||u(x)| - |u(y)||\leq \left|u(x)-e^{\imath (x-y)\cdot A\left(\frac{x+y}{2}\right)}u(y)\right|,
%\]
\begin{lem}\label{DI}
If $u\in H_{\e}^s$, then $|u|\in H^s(\R^3,\R)$ and the following fractional diamagnetic inequality holds
\begin{equation*}
%\label{eqDI}
[|u|]^2 \leq [u]^{2}_{A_{\e}}
\end{equation*}
where
$$
[u]^2 := \frac{c_{3,s}}{2}\iint_{\R^{6}} \frac{|u(x)- u(y)|^{2}}{|x-y|^{3+2s}} \, dxdy.
$$
\end{lem}
%In view of \cite[Lemma 2.3]{AD} we also have: 
\begin{lem}\label{Sembedding}
	The space $H^{s}_{\e}$ is continuously embedded in $L^{r}(\R^{3}, \C)$ for all $r\in [2, 2^{*}_{s}]$, and compactly embedded in $L_{\rm loc}^{r}(\R^{3}, \C)$ for all $r\in [1, 2^{*}_{s})$.\\
	Moreover, if $V_\infty=\infty$, then, for any bounded sequence $(u_{n})$  in $\h$, we have that, up to a subsequence, $(|u_{n}|)$ is strongly convergent in $L^{r}(\R^{3}, \R)$ for all $r\in [2, 2^{*}_{s})$.
\end{lem}
%As proved in \cite[Lemma 2.4]{AD}, we know that:
\begin{lem}\label{aux}
If $u\in H^{s}(\R^{3}, \R)$ and $u$ has compact support, then $w=e^{\imath A(0)\cdot x} u \in \h$.
\end{lem}

\noindent
Now, let $u\in \h$, and we define
\begin{equation}\label{eqtruncation}
\hat{u}_{j}(x):=\varphi_{j}(x)u(x)
\end{equation}
where $j\in\mathbb{N}^*$ and $\varphi_{j}(x)=\varphi(2x/j)$ with $\varphi\in C^{\infty}_{0}(\R^{3}, \R)$, $0\leq \varphi\leq 1$, $\varphi(x)=1$ if $|x|\leq 1$, and $\varphi(x)=0$ if $|x|\geq 2$. Note that $\hat{u}_{j}\in \h$ and $\hat{u}_{j}$ has compact support. Moreover, from \cite[Lemma $3.2$]{ZSZ} it follows that
\begin{lem}\label{truncation}
	For any $\e>0$, it holds $\|\hat{u}_{j}-u\|_{\e}\rightarrow 0$ as $j\rightarrow \infty$.
\end{lem}

\noindent
Moreover, we recall the following useful result proved in \cite[Lemma 3.1]{ZSZ}:
\begin{lem}\label{vanishing}
	Let $\tau\in [2, 2^{*}_{s})$ and $(u_{n})\subset \h$ be a bounded sequence. Then there exists a subsequence $(u_{n_{j}})\subset \h$ such that for any $\sigma>0$ there exists $r_{\sigma,\tau}>0$ such that 
	\begin{equation}\label{DL}
	\limsup_{j\rightarrow \infty} \int_{B_{j}(0)\setminus B_{r}(0)} |u_{n_{j}}|^{\tau} dx\leq \sigma
	\end{equation}
	for any $r\geq r_{\sigma}$.
\end{lem}

\noindent
In view of $(f_1)$-$(f_2)$ and arguing as in \cite[Lemma 2.7]{AD}, we can prove the following properties for the nonlinearity:
\begin{lem}
\label{propf}
Assume that $(f_1)$-$(f_4)$ hold. Then $f$ satisfies the following properties:
\begin{enumerate}[label=(\roman*),ref=\roman*]
	\item \label{propf1} for every $\xi >0$ there exists $C_{\xi}>0$ such that for all $t\in\R$,
	$$\frac{\theta}{2}F(t^{2}) \leq f(t^2)t^2\leq \xi t^{4} + C_{\xi} |t|^{q};$$
	\item \label{propf2} there exist $C_1,C_2>0$ such that for all $t\in\R$, $F(t^{2})\geq C_{1}|t|^{\theta} - C_{2}$;
	\item \label{propf3} if $u_{n_j}\rightharpoonup u$ in $\h$ and $\hat{u}_j$ is defined as in \eqref{eqtruncation} we have that
	$$
	\int_{\R^{3}} F(|u_{n_j}|^{2})-F(|u_{n_j}-\hat{u}_{j}|^{2})-F(|\hat{u}_{j}|^{2}) dx=o_{j}(1)
	\quad
	\hbox{as }j\to\infty,
	$$
	where $o_{j}(1)\rightarrow 0$ as $j\rightarrow \infty$;
	\item \label{propf4} if $(u_n)\subset\h$ is bounded, $(u_{n_j})$ a subsequence as in Lemma \ref{vanishing} such that $u_{n_j}\rightharpoonup u$ in $\h$ and $\hat{u}_j$ is defined as in \eqref{eqtruncation} we have that
	\[
\Re\left(\int_{\R^{3}} [f(|u_{n_{j}}|^{2})u_{n_{j}}-f(|u_{n_j}-\hat{u}_{j}|^{2})(u_{n_j}-\hat{u}_{j})-f(|\hat{u}_{j}|^{2})\hat{u}_{j}] \bar{\varphi} dx\right)\rightarrow 0
\quad
\hbox{as }j\to\infty
\]
uniformly with respect to $\varphi\in \h$ with $\|\varphi\|_{\e}\leq 1$.
\end{enumerate}
\end{lem}

\noindent
Now, let $s, t\in (0, 1)$ be such that $4s+2t\geq 3$. Recalling the embedding $H^{s}(\R^{3}, \R)\subset L^{q}(\R^{3}, \R)$ for all $q\in [2, \2)$ (see \cite[Theorem 6.5]{DPV}), we get
\begin{equation}\label{ruiz}
H^{s}(\R^{3}, \R)\subset L^{\frac{12}{3+2t}}(\R^{3}, \R).
\end{equation}
Fix $u\in \h$. By Lemma \ref{DI} we know that $|u|\in H^{s}(\R^{3}, \R)$. Now, let us define the functional $\mathcal{L}_{|u|}: D^{t, 2}(\R^{3}, \R)\rightarrow \R$ given by
$$
\mathcal{L}_{|u|}(v)=\int_{\R^{3}} |u|^{2} v\, dx,
$$
where $D^{t, 2}(\R^{3}, \R)=\{u\in L^{2^{*}_{t}}(\R^{3}, \R): [u]<\infty\}$.
Using H\"older inequality and \eqref{ruiz} we can see that
\begin{equation}
|\mathcal{L}_{|u|}(v)|\leq \left(\int_{\R^{3}} |u|^{\frac{12}{3+2t}}dx\right)^{\frac{3+2t}{6}}  \left(\int_{\R^{3}} |v|^{2^{*}_{t}}dx\right)^{\frac{1}{2^{*}_{t}}}\leq C\|u\|^{2}_{D^{s,2}}\|v\|_{D^{t,2}},
\end{equation}
where
$$
\|v\|^{2}_{D^{t,2}}=\iint_{\R^{6}} \frac{|v(x)-v(y)|^{2}}{|x-y|^{3+2t}}dxdy.
$$
Then, $\mathcal{L}_{|u|}$ is a linear continuous functional, and applying the Lax-Milgram Theorem we can find a unique $\phi_{|u|}^{t}\in D^{t, 2}(\R^{3}, \R)$ such that 
\begin{equation}\label{FPE}
(-\Delta)^{t} \phi_{|u|}^{t}= |u|^{2} \mbox{ in } \R^{3},
\end{equation}
which can be expressed via the following $t$-Riesz formula
\begin{equation}\label{RF}
\phi_{|u|}^{t}(x)=c_{t}\int_{\R^{3}} \frac{|u(y)|^{2}}{|x-y|^{3-2t}} \, dy  \quad (x\in \R^{3}), \quad c_{t}=\pi^{-\frac{3}{2}}2^{-2t}\frac{\Gamma(3-2t)}{\Gamma(t)}.
\end{equation}
In the sequel, we will omit the constant $c_{t}$ for convenience in \eqref{RF}. Now we prove the following properties of the function $\phi_{|u|}^{t}$.
\begin{lem}\label{poisson}
Let us assume that $4s+2t\geq 3$ and $u\in \h$.
Then we have:
\begin{enumerate}
\item 
%$\|\phi^{t}_{|u|}\|_{D^{t,2}}\leq C\|u\|_{L^{\frac{12}{3+2t}}(\R^{3})}^{2}$, $\int_{\R^{3}} \phi_{|u|}|u|^{2}dx\leq C\|u\|^{4}_{L^{\frac{12}{3+2t}}(\R^{3})}$ and 
$\phi_{|u|}^{t}: H^{s}(\R^{3},\R)\rightarrow D^{t,2}(\R^{3}, \R)$ is continuous and maps bounded sets into bounded sets,
\item if $u_{n}\rightharpoonup u$ in $\h$ then $\phi_{|u_{n}|}^{t}\rightharpoonup \phi_{|u|}^{t}$ in $D^{t,2}(\R^{3},\R)$,
\item $\phi^{t}_{|ru|}=r^{2}\phi^{t}_{|u|}$ for all $r\in \R$ and $\phi^{t}_{|u(\cdot+y)|}(x)=\phi^{t}_{|u|}(x+y)$,
\item $\phi_{|u|}^{t}\geq 0$ for all $u\in \h$, and we have
$$
\|\phi_{|u|}^{t}\|_{D^{t,2}}\leq C\|u\|_{L^{\frac{12}{3+2t}}(\R^{3})}^{2}\leq C\|u\|^{2}_{\e} \, \mbox{ and } \int_{\R^{3}} \phi_{|u|}^{t}|u|^{2}dx\leq C\|u\|^{4}_{L^{\frac{12}{3+2t}}(\R^{3})}\leq C\|u\|_{\e}^{4}.
$$
%\item if $u_{n_j}\rightharpoonup u$ in $\h$ and $\hat{u}_j$ is defined as in \eqref{eqtruncation} we have that
%$$
%\Psi(u_{n_j})-\Psi(u_{n_j}-\hat{u}_{j})-\Psi(\hat{u}_{j}) =o_{j}(1)
%\quad
%\hbox{as }j\to\infty;
%$$
%where $\Psi(u)=\int_{\R^{3}} \phi^{t}_{u} |u|^{2}dx$.
%\item if $(u_n)\subset\h$ is bounded, $(u_{n_j})$ a subsequence as in Lemma \ref{vanishing} such that $u_{n_j}\rightharpoonup u$ in $\h$ and $\hat{u}_j$ is defined as in \eqref{eqtruncation} we have that
%$$
%\langle \Psi'(u_{n_j})-\Psi'(u_{n_j}-\hat{u}_{j})-\Psi'(\hat{u}_{j}), \phi\rangle=o_{j}(1)
%\quad
%\hbox{as }j\to\infty
%$$
%uniformly with respect to $\phi\in \h$ with $\|\phi\|_{\e}\leq 1$.
\end{enumerate} 
\end{lem}
\begin{proof}
$(1)$ Since $\phi_{|u|}^{t}\in D^{t,2}(\R^{3}, \R)$ satisfies \eqref{FPE}, we have
\begin{equation}\label{spput}
\int_{\R^{3}}(-\Delta)^{\frac{t}{2}}\phi_{|u|}^{t} (-\Delta)^{\frac{t}{2}}v \,dx=\int_{\R^{3}} |u|^{2}v \,dx
\end{equation}
for all $v\in D^{t,2}(\R^{3}, \R)$. Then $\|\mathcal{L}_{|u|}\|_{\mathcal{L}(D^{t,2}, \R)}=\|\phi_{|u|}^{t}\|_{D^{t,2}}$ for all $u\in \h$ (here $\mathcal{L}(D^{t,2}, \R)$ is the space of bounded linear operators from
$D^{t,2}$ into $\R$), and our claim is to prove that  $u\mapsto \mathcal{L}_{|u|}$ is continuous to deduce that $\phi_{|u|}^{t}$ is continuous.
Let $u_{n}\rightarrow u$ in $\h$. In view of Lemma \ref{DI} and Theorem \ref{Sembedding} we have $|u_{n}|\rightarrow |u|$ in $L^{\frac{12}{3+2t}}(\R^{3})$. Thus, for all $v\in D^{t,2}(\R^{3}, \R)$, we can see that
\begin{align*}
|\mathcal{L}_{|u_{n}|}(v)-\mathcal{L}_{|u|}(v)|&=\left|\int_{\R^{3}} (|u_{n}|^{2}-|u|^{2})v\, dx\right| \\
&\leq \left(\int_{\R^{3}} ||u_{n}|^{2}-|u|^{2}|^{\frac{6}{3+2t}} \, dx\right)^{\frac{3+2t}{6}} \|v\|_{L^{\frac{6}{3-2t}}(\R^{3})} \\
&\leq C\left[\left(\int_{\R^{3}} ||u_{n}|-|u||^{\frac{12}{3+2t}} \, dx\right)^{\frac{1}{2}} \left(\int_{\R^{3}} ||u_{n}|+|u||^{\frac{12}{3+2t}} \, dx\right)^{\frac{1}{2}}\right]^{\frac{3+2t}{6}} \|v\|_{D^{t,2}} \\
&\leq C  \||u_{n}|-|u|\|_{L^{\frac{12}{3+2t}}(\R^{3})}\|v\|_{D^{t,2}}\rightarrow 0 \quad \mbox{ as } n\rightarrow \infty.
\end{align*}
%which implies that $\|\phi_{|u_{n}|}^{t}-\phi_{|u|}^{t} \|_{D^{t,2}}=\|\mathcal{L}_{|u_{n}|}-\mathcal{L}_{|u|}\|_{\mathcal{L}(D^{t,2}, \R)}\rightarrow 0$ as $n\rightarrow \infty$.\\
$(2)$ Let $u_{n}\rightharpoonup u$ in $\h$ and fix $v\in C^{\infty}_{c}(\R^{3}, \R)$.
In the light of Lemma \ref{DI} and Theorem \ref{Sembedding} we get $|u_{n}|\rightarrow |u|$ in $L^{q}_{loc}(\R^{3}, \R)$ for all $q\in [1, \2)$. Therefore
\begin{align*}
\langle \phi_{|u_{n}|}^{t}-\phi_{|u|}^{t}, v\rangle&=\int_{\R^{3}} (|u_{n}|^{2}-|u|^{2})v\, dx \\
&\leq \left(\int_{supp(v)} ||u_{n}|-|u||^{2}\, dx\right)^{\frac{1}{2}} \left(\int_{\R^{3}} ||u_{n}|+|u||^{2}\, dx\right)^{\frac{1}{2}}  \|v\|_{L^{\infty}(\R^{3})}\\
&\leq C \||u_{n}|-|u|\|_{L^{2}(supp(v))} \|v\|_{L^{\infty}(\R^{3})}\rightarrow 0,
\end{align*}
and the conclusion follows by a density argument.\\
$(3)$ is obtained by the definition of $\phi_{|u|}^{t}$.\\
$(4)$ It is clear that $\phi_{|u|}^{t}\geq 0$. 
Using \eqref{spput} with $v=\phi_{|u|}^{t}$, H\"older inequality and \eqref{ruiz} we have
%Sobolev embedding \cite[Theorem 6.5]{DPV} we have
\begin{align*}
\|\phi_{|u|}^{t}\|^{2}_{D^{t,2}}&\leq \|u\|^{2}_{L^{\frac{12}{3+2t}}(\R^{3})} \|\phi_{|u|}^{t}\|_{L^{2^{*}_{t}}(\R^{3})}
\leq C  \|u\|^{2}_{L^{\frac{12}{3+2t}}(\R^{3})} \|\phi_{|u|}^{t}\|_{D^{t,2}}\leq C \|u\|^{2}_{\e}  \|\phi_{|u|}^{t}\|_{D^{t,2}}.
\end{align*}
On the other hand, in view of  \eqref{RF}, Hardy-Littlewood-Sobolev inequality \cite[Theorem 4.3]{LL} and \eqref{ruiz} we get
%Sobolev embedding \cite[Theorem 6.5]{DPV} we get
\begin{align*}
\int_{\R^{3}} \phi_{|u|}^{t}|u|^{2}dx&\leq C \||u|^{2}\|^{2}_{L^{\frac{6}{3+2t}}(\R^{3})}=C \|u\|^{4}_{L^{\frac{12}{3+2t}}(\R^{3})}\leq C\|u\|_{\e}^{4}. 
\end{align*}
\end{proof}

\noindent
In the lemma below we prove a Brezis-Lieb splitting property \cite{BL} (see also \cite{Ack, MVS, ZZ}) for the following operator 
$$
\Psi: \h\rightarrow \R \quad \quad \Psi(u):=\int_{\R^{3}} \phi^{t}_{|u|} |u|^{2}dx
$$
and its differential $\Psi'$.
These results will be useful to study the decomposition of the functional associated with \eqref{Pe} along $(PS)$ sequences; see Proposition \ref{prop2.1} in Section $4$.
\begin{lem}\label{CCS}
Let us assume that $2s+2t> 3$. 
%and we define $\Psi(u)=\int_{\R^{3}} \phi^{t}_{|u|} |u|^{2}dx$.
Then we have the following splittings:
\begin{compactenum}[$(i)$]
\item if $u_{n_j}\rightharpoonup u$ in $\h$ and $\hat{u}_j$ is defined as in \eqref{eqtruncation} we have that
	$$
	\Psi(u_{n_j})-\Psi(u_{n_j}-\hat{u}_{j})-\Psi(\hat{u}_{j}) =o_{j}(1)
	\quad
	\hbox{as }j\to\infty;
	$$
	\item if $(u_n)\subset\h$ is bounded, $(u_{n_j})$ a subsequence as in Lemma \ref{vanishing} such that $u_{n_j}\rightharpoonup u$ in $\h$ and $\hat{u}_j$ is defined as in \eqref{eqtruncation} we have that
	$$
	\langle \Psi'(u_{n_j})-\Psi'(u_{n_j}-\hat{u}_{j})-\Psi'(\hat{u}_{j}), \varphi\rangle=o_{j}(1)
	\quad
\hbox{as }j\to\infty
$$
uniformly with respect to $\varphi\in \h$ with $\|\varphi\|_{\e}\leq 1$.
\end{compactenum} 
\end{lem}
\begin{proof}
%The proofs of these splittings can be obtained as in Lemma $3.5$ in \cite{Ack}. Anyway, 
The verification of $(i)$ is similar to and simpler than that of $(ii)$, so we only check the latter.
Combining $u_{n_{j}}\rightharpoonup u$ in $\h$, Lemma \ref{Sembedding}, Lemma \ref{truncation} and Lemma \ref{poisson}, we can see
that for any $r>0$
$$
\lim_{j\rightarrow \infty} \left|\int_{B_{r}(0)} (\phi_{|u_{n_{j}}|}^{t}u_{n_{j}} -\phi_{|u_{n_{j}}-\hat{u}_{j}|}^{t}(u_{n_{j}}-\hat{u}_{j})-\phi_{|\hat{u}_{j}|}^{t}\hat{u}_{j})\bar{\varphi} \, dx\right|=0
$$
uniformly with respect to $\varphi\in \h$ with $\|\varphi\|_{\e}\leq 1$.\\
On the other hand, from Lemma \ref{DL}, it follows that for any $\sigma>0$ there exists $r_{\sigma}>0$ such that
\begin{equation*}
\limsup_{j\rightarrow \infty} \int_{B_{j}(0)\setminus B_{r}(0)} |u_{n_{j}}|^{2} dx\leq \sigma
\end{equation*}
for any $r\geq r_{\sigma}$.
Then
\begin{equation*}
\limsup_{j\rightarrow \infty} \int_{B_{j}(0)\setminus B_{r}(0)} |\tilde{u}_{j}|^{2} dx\leq \int_{\R^{3}\setminus B_{r}(0)}  |u|^{2} dx \leq \sigma
\end{equation*}
for any $r\geq r_{\sigma}$.	
Now, since $2s+2t>3$, we can find $\frac{3}{2s}<p<\frac{3}{3-2t}$ so that $2p'\in (2, \2)$. Moreover, taking $q>\frac{3}{3-2t}$ and using again $2s+2t>3$, we obtain that  $2q'\in (2, \2)$. 
Then, applying H\"older inequality, we have for all $u\in \h$ 
\begin{align*}
\phi_{|u|}^{t}(x)=\int_{\R^{3}} \frac{|u(y)|^{2}}{|x-y|^{3-2t}} dy&\leq \|u\|^{2}_{L^{2p'}(B_{1}(x))}\left(\int_{|y-x|<1} \frac{|u(y)|^{2}}{|x-y|^{p(3-2t)}} dy\right)^{\frac{1}{p}} \\
&+\|u\|^{2}_{L^{2q'}(B^{c}_{1}(x))}\left(\int_{|y-x|>1} \frac{|u(y)|^{2}}{|x-y|^{q(3-2t)}} dy\right)^{\frac{1}{q}}\\
&\leq C\max\left\{\|u\|^{2}_{L^{2p'}(\R^{3})}, \|u\|^{2}_{L^{2q'}(\R^{3})}\right\}
\end{align*}
for some $C>0$ independent of $x$.
Fix $\varphi\in \h$ such that $\|\varphi\|_{\e}\leq 1$. Taking into account the boundedness of $(u_{n_{j}})$ and $(\hat{u}_{j})$ in $\h$ and using Lemma \ref{truncation}, we can see that the above estimate yields
\begin{align*}
&\limsup_{j\rightarrow \infty} \left|\int_{\R^{3}} (\phi_{|u_{n_{j}}|}^{t}u_{n_{j}} -\phi_{|u_{n_{j}}-\hat{u}_{j}|}^{t}(u_{n_{j}}-\hat{u}_{j})-\phi_{|\hat{u}_{j}|}^{t}\hat{u}_{j})\bar{\varphi} \, dx\right|\\
&\leq\limsup_{j\rightarrow \infty} \int_{B_{j}(0)\setminus B_{r}(0)} \left|\phi_{|u_{n_{j}}|}^{t}u_{n_{j}} -\phi_{|u_{n_{j}}-\hat{u}_{j}|}^{t}(u_{n_{j}}-\hat{u}_{j})-\phi_{|\hat{u}_{j}|}^{t}\hat{u}_{j} \right| |\bar{\varphi}| \, dx \\
&\leq C \limsup_{j\rightarrow \infty} \Bigl[\left(\|u_{n_{j}}\|_{L^{2}(B_{j}(0)\setminus B_{r}(0))}+\|\hat{u}_{j}\|_{L^{2}(B_{j}(0)\setminus B_{r}(0))}  \right) \|\varphi\|_{L^{2}(\R^{3})} \\
&\times \max\left\{\|u_{n_{j}}\|^{2}_{L^{2p'}(\R^{3})}, \|u_{n_{j}}\|^{2}_{L^{2q'}(\R^{3})}, \|\hat{u}_{j}\|^{2}_{L^{2p'}(\R^{3})}, \|\hat{u}_{j}\|^{2}_{L^{2q'}(\R^{3})}\right\}\Bigr]\\
%&\leq C \left[\limsup_{j\rightarrow \infty} \|u_{n_{j}}\|_{L^{2}(B_{j}\setminus B_{r})}+\|u\|_{L^{2}(\R^{3}\setminus B_{r})}\right] \\
&\leq C\sigma^{1/2}.
\end{align*}
From the arbitrariness of $\sigma>0$ we get the thesis.
%The proofs of $(5)$ and $(6)$ can be obtained arguing as in Lemma $3.5$ in \cite{Ack} and Lemma $3.4$ in \cite{CCS}. 
\end{proof}
\begin{remark}
In order to lighten the notation, in what follows we neglect the constant $c_{3,s}$ appearing in the definition of $[\cdot]_{A}$.
\end{remark}

\section{Functional setting}
%Our aim is to find solutions of \eqref{Pe} in the sense of the following definition.
%\begin{defn}
%	We say that $u\in H^{s}_{\e}$ is a weak solution to \eqref{Pe} if for any $v\in H^{s}_{\e}$
%	\begin{align*}
%	& \Re \Bigl(\frac{c_{3,s}}{2} \iint_{\R^{6}} \frac{(u(x)-e^{\imath (x-y)\cdot A_{\e}(\frac{x+y}{2})} u(y)) \overline{(v(x)-e^{\imath (x-y)\cdot A_{\e}(\frac{x+y}{2})} v(y))}}{|x-y|^{3+2s}} \, dxdy \\
%	&\qquad +\int_{\R^{3}} V(\e x) u \bar{v}\, dx-\int_{\R^{3}} f(|u|^{2})u\bar{v}\, dx\Bigr)=0.
%	\end{align*}
%\end{defn}
In order to find weak solutions to \eqref{Pe}, we look for  critical points of the  functional $J_{\e}: \h\rightarrow \R$ associated with \eqref{Pe}  defined by 
$$
J_{\e}(u)=\frac{1}{2} \|u\|^{2}_{\e}+\frac{1}{4} \int_{\R^{3}} \phi_{|u|}^{t}|u|^{2}\, dx-\frac{1}{2} \int_{\R^{3}} F(|u|^{2})\, dx.
$$
In view of Lemma \ref{Sembedding} and Lemma \ref{propf}, it is easy to check that $J_{\e}$ is well-defined, $J_{\e}\in C^{1}(\h, \R)$ and its differential is given by
\begin{align*}
	\langle J'_{\e}(u), v\rangle=& \Re \Bigl( \iint_{\R^{6}} \frac{(u(x)-e^{\imath (x-y)\cdot A_{\e}(\frac{x+y}{2})} u(y)) \overline{(v(x)-e^{\imath (x-y)\cdot A_{\e}(\frac{x+y}{2})} v(y))}}{|x-y|^{3+2s}} \, dxdy \\
	&\qquad +\int_{\R^{3}} V_{\e} (x) u \bar{v}\, dx-\int_{\R^{3}} f(|u|^{2})u\bar{v}\, dx\Bigr).
	\end{align*}
Hence, the critical points of $J_{\e}$ are exactly the weak solutions of \eqref{Pe}.	
Now we show that, for any $\e>0$, the functional $J_{\e}$ possesses a mountain pass geometry \cite{AR}.
\begin{lem}\label{MPG}
	The functional $J_{\e}$ satisfies the following conditions:
	\begin{compactenum}[$(i)$]
		\item  there exist $\alpha, \rho >0$ such that $J_{\e}(u)\geq \alpha$ with $\|u\|_{\e}=\rho$; 
		\item  there exists $e\in \h$ with $\|e\|_{\e}>\rho$ such that $J_{\e}(e)<0$. 
	\end{compactenum}
\end{lem}

\begin{proof}
$(i)$ Using Lemma \ref{propf}-$(i)$, Lemma \ref{poisson}-$(4)$ and Lemma \ref{Sembedding}, for $\xi$ sufficiently small we have
\begin{equation*}
J_{\e}(u) 
\geq \frac{1}{2} \|u\|_{\e}^{2}-\frac{\xi}{4}\|u\|_{L^{4}(\R^{3})}^{4}-C_{\xi}\|u\|^{q}_{L^{q}(\R^{3})}
\geq \frac{1}{2} \|u\|_{\e}^{2} - \xi C\|u\|^{4}_{\e}-C_2 \|u\|_{\e}^{q}.
\end{equation*}
$(ii)$ In view of Lemma \ref{propf}-$(ii)$ and recalling that $\theta>4$, we can see that for any $u \in C^{\infty}_{c}(\R^{3}, \C)$ such that $u \not \equiv 0$, we obtain 
\begin{align*}
J_{\e}(T u)&=\frac{T^{2}}{2} \|u\|_{\e}^{2}+\frac{T^{4}}{4} \int_{\R^{3}} \phi_{|u|}^{t} |u|^{2}dx-\frac{1}{2}\int_{\R^{3}} F(T^{2}|u|^{2})dx\\
&\leq \frac{T^{4}}{2}  \left(\|u\|_{\e}^{2}+\int_{\R^{3}} \phi_{|u|}^{t} |u|^{2}dx\right)-CT^{\vartheta} \|u\|^{\theta}_{L^{\theta}(\R^{3})}+C <0
\end{align*}
for $T>0$ large enough.
\end{proof}

In view of Lemma \ref{MPG}, we can use the Ekeland Variational Principle to see that there exists a $(PS)_{c_{\e}}$ sequence $(u_{n})\subset \h$, that is 
\begin{equation}\label{PSc}
J_{\e}(u_{n})\rightarrow c_{\e} \quad \mbox{ and }\quad J'_{\e}(u_{n})\rightarrow 0, 
\end{equation}
where $c_{\e}$ is the minimax level of the mountain pass theorem, that is
$$
c_{\e}:= \inf_{\gamma\in\Gamma_{\e}} \max_{t\in [0,1]} J_{\e}(\gamma(t))
$$
where 
$$
\Gamma_{\e}:=\{\gamma\in C([0,1],H_{\e}^s): \gamma(0)=0,J_{\e}(\gamma(1))<0\}.
$$
Moreover, we can see that the following assertion holds:
\begin{lem}\label{bPS}
If $(u_{n})$ is a $(PS)_{c_{\e}}$ sequence then $(u_{n})$ bounded in $\h$.
\end{lem}
\begin{proof}
In view of \eqref{PSc} we can see that
\begin{align*}
c_{\e}+o(1)&=\frac{1}{2} \|u_{n}\|_{\e}^{2}+\frac{1}{4} \int_{\R^{3}} \phi^{t}_{|u_{n}|} |u_{n}|^{2}dx-\frac{1}{2}\int_{\R^{3}} F(|u_{n}|^{2}) dx
\end{align*}
and
\begin{equation*}
o(1)=\|u_{n}\|_{\e}^{2}+\int_{\R^{3}} \phi^{t}_{|u_{n}|} |u_{n}|^{2}dx-\int_{\R^{3}} f(|u_{n}|^{2})|u_{n}|^{2} dx.
\end{equation*}
Then, using $(f_3)$ we can deduce that
\begin{align*}
c_{\e}+o(1)&=\frac{1}{4}\|u_{n}\|^{2}_{\e}+\frac{1}{4}\int_{\R^{3}} f(|u_{n}|^{2})|u_{n}|^{2}-2F(|u_{n}|^{2}) dx\\
&\geq \frac{1}{4}\|u_{n}\|^{2}_{\e}+\left(\frac{\vartheta-4}{8}\right)\int_{\R^{3}} F(|u_{n}|^{2}) dx \\
&\geq \frac{1}{4}\|u_{n}\|^{2}_{\e} 
\end{align*}
which implies that $(u_{n})$ is bounded in $\h$.
\end{proof}

%Let us observe that $(u_{n})$ is bounded in $\h$. In fact by using \eqref{PSc} and $(f_3)$ we can see that
%\begin{align*}
%c+o_{n}(1)\|u_{n}\|_{\e}&= J_{\e}(u_{n})-\frac{1}{\theta}\langle J'_{\e}(u_{n}), u_{n}\rangle \\
%&= \left(\frac{1}{2}-\frac{1}{\theta}\right)\|u_{n}\|^{2}_{\e}+\left( \frac{1}{4}-\frac{1}{\theta}\right)\int_{\R^{3}} \phi_{|u_{n}|}^{t}|u_{n}|^{2}dx\\
%&+\frac{1}{\theta}\int_{\R^{3}} \left[f(|u_{n}|^{2})|u_{n}|^{2}-\frac{\theta}{2} F(|u_{n}|^{2})\right]\, dx\\
%&\geq \left(\frac{1}{2}-\frac{1}{\theta}\right)\|u_{n}\|^{2}_{\e}
%\end{align*}
As in \cite[Chapter 4]{W}, it is easy to see that $c_{\e}$ can be characterized as follows:
\begin{equation*}
c_{\e}= \inf_{u\in \h\setminus \{0\}} \sup_{t\geq 0} J_{\e}(tu) = \inf_{u\in \N_{\e}} J_{\e}(u), 
\end{equation*}
where
$$
\N_{\e}:=\{u\in \h\setminus\{0\}: \langle J'_{\e}(u),u\rangle=0\}
$$
is the Nehari manifold associated to $J_{\e}$.
Moreover, we have the following properties.
\begin{lem}\label{LemNeharyE}
We have:
\begin{compactenum}[$(i)$]
	\item there exists $K>0$ such that,  for all $u\in \N_{\e}$, $\|u\|_{\e}\geq K$;
	\item  for any $u\in \h\setminus \{0\}$ there exists a unique $t_{0}= t_{0}(u)$ such that $J_{\e}(t_{0}u)= \max_{t\geq 0} J_{\e}(tu)$ and then $t_{0}u\in \N_{\e}$. 
\end{compactenum}
\end{lem}
\begin{proof}
$(i)$ Fix $u\in \N_{\e}$. In view of Lemma \ref{propf}-$(i)$ and Lemma \ref{poisson}-$(4)$ we can obtain
\begin{align*}
0&=\|u\|_{\e}^{2}+\int_{\R^{3}} \phi_{|u|}|u|^{2}\, dx-\int_{\R^{3}} f( |u|^{2})|u|^{2}\,dx\\
&\geq \|u\|_{\e}^{2}-\e \|u\|_{L^{4}(\R^{3})}^{4}-C\|u\|_{L^{q}(\R^{3})}^{q} \\
&\geq \|u\|_{\e}^{2}-\e C\|u\|_{\e}^{4}-C\|u\|_{\e}^{q}
\end{align*}
which implies that there exists $K>0$ such that $\|u\|_{\e}\geq K$.

$(ii)$ Take $u\in \h\setminus \{0\}$ and set $h(t):=J_{\e}(tu)$ for $t\geq 0$. 
From the arguments in Lemma \ref{MPG}, we can see that $h(0)=0$, $h(t)>0$ for $t>0$ small and $h(t)<0$ for $t$ large. Then there exists $t_{u}>0$ such that $h(t_{u})=\max_{t\geq 0} h(t)$ so that $h'(t_{u})=0$ and $t_{u}u\in \mathcal{N}_{\e}$.  In order to prove the uniqueness of a such $t_{u}$, let $0<t_{u}<t'_{u}$ such that $t_{u}u, t'_{u}u\in \mathcal{N}_{\e}$. Then we have
\begin{align*}
\left(\frac{1}{(t'_{u})^{2}}-\frac{1}{(t_{u})^{2}}\right) \|u\|_{\e}^{2}=\int_{\R^{3}} \left[\frac{f((t'_{u})^{2}|u|^{2})}{(t'_{u})^{2}|u|^{2}}-\frac{f((t_{u})^{2}|u|^{2})}{(t_{u})^{2}|u|^{2}} \right]|u|^{4}dx.
\end{align*}
Using $(f_4)$ we can deduce that the above equation makes no sense.\end{proof}

We will see that it is very important to compare $c_{\e}$ with the minimax level of the autonomous problem
\begin{equation}
\label{AP}
\tag{$P_{\mu}$}
(-\Delta)^{s} u + \mu u+\phi^{t}_{u}u =  f(u^{2})u  \mbox{ in } \R^{3},
\end{equation}
with $\mu>0$, whose solutions can be obtained as critical points of the functional $I_{\mu}: H^{s}_{\mu}\rightarrow \R$ given by
$$
I_{\mu}(u)
=\frac{1}{2} \|u\|^{2}_{\mu}+\frac{1}{4}\int_{\R^{3}} \phi_{u}^{t}u^{2}\, dx-\frac{1}{2} \int_{\R^{3}} F(u^{2})\, dx,
$$
where $H^{s}_{\mu}$ is the space $H^{s}(\R^{3}, \R)$ endowed with the norm
$$
\|u\|^{2}_{\mu}:=[u]^2+\mu \|u\|_{L^{2}(\R^{3})}^{2}.
$$
We also define the Nehari manifold associated to \eqref{AP}
$$
\mathcal{M}_{\mu}=\{u\in H^{s}_{\mu}: \langle I'_{\mu}(u),u\rangle=0\}
$$
and 
$$
m_{\mu}
%:= \inf_{\gamma\in H^{s}_{\mu}} \max_{t\in [0,1]} I_{\mu}(\gamma(t))
%=\inf_{u\in H^{s}_{\mu}\setminus \{0\}} \sup_{t\geq 0} I_{\mu}(t u)
=\inf_{u\in \mathcal{M}_{\mu}} I_{\mu}(u).
$$
%with $\mathcal{N}_\mu:=\{\gamma\in C([0,1],H^s(\R^3,\R)) : \gamma(0)=0,I_\mu(\gamma(1))<0\}$.\\
%We will call {\em ground state} for \eqref{Plim} each minimum of $I_\mu$ in $\mathcal{M}_{\mu}$, wich is also a solution of \eqref{Plim}.

\begin{remark}
Arguing as in Lemma \ref{LemNeharyE} we can prove that for every fixed $\mu>0$ there exists $K>0$ such that,  for all $u\in \mathcal{M}_{\mu}$, $\|u\|_{\mu}\geq K$ and that for any $u\in H^{s}_{\mu}\setminus \{0\}$ there exists a unique $t_{0}= t_{0}(u)$ such that $I_{\mu}(t_{0}u)= \max_{t\geq 0} I_{\mu}(tu)$ and then $t_{0}u\in \mathcal{M}_{\mu}$.
\end{remark}

\noindent
In order to prove that $m_{\mu}$ can be achieved, we first recall the following useful lemma \cite[Lemma 2.2]{FQT}.
\begin{lem}
%\cite[Lemma 2.2]{FQT}
	\label{Lions}
	Let $q\in [2, 2^{*}_{s})$. If $(u_{n})$ is a bounded sequence in $H^{s}(\R^{3}, \R)$ and  
	\begin{equation*}
	\lim_{n\rightarrow \infty} \sup_{y\in \R^{3}} \int_{B_{R}(y)} |u_{n}|^{q} dx=0
	\end{equation*}
	for some $R>0$, then $u_{n}\rightarrow 0$ in $L^{r}(\R^{3}, \R)$ for all $r\in (q, 2^{*}_{s})$.
\end{lem}

\noindent
At this point we can prove the following result.
\begin{lem}\label{compactness}
Let $d\in\R$ and $(u_{n})\subset H^{s}_{\mu}$ be a $(PS)_{d}$ sequence for $I_{\mu}$.
%such that $u_{n}\rightharpoonup 0$ in $\h$. 
Then, one of the following alternatives occurs:
\begin{compactenum}[$(i)$]
\item $u_{n}\rightarrow 0$ in $H^{s}_{\mu}$; 
\item  there are a sequence $(y_{n})\subset \R^{3}$ and constants $R, \beta>0$ such that 
\begin{equation*}
\liminf_{n\rightarrow \infty} \int_{B_{R}(y_{n})} |u_{n}|^{4} dx \geq \beta >0. 
\end{equation*}
\end{compactenum}
\end{lem}
\begin{proof}
	Suppose that $(ii)$ does not hold true. Then, for every $R>0$ we have
	\begin{equation*}
	\lim_{n\rightarrow \infty} \sup_{y\in \R^{3}} \int_{B_{R}(y)} |u_{n}|^{4} dx=0. 
	\end{equation*}
	Arguing as in Lemma \ref{bPS}, we can see that $(u_{n})$ is bounded in $H^{s}_{\mu}$. In view of Lemma \ref{Lions}, we infer that $\|u_{n}\|_{L^{r}(\R^{3})}\rightarrow 0$ for all $r\in (2, 2^{*}_{s})$.
	%from Lemma \ref{DI} we deduce that $(|u_{n}|)$ is bounded in $H^{s}(\R^{3}, \R)$, so by Lemma \ref{Lions} it follows that $	|u_{n}|_{t}\rightarrow 0$ for all $t\in (4, 2^{*}_{s})$.
	This and $(f_1)$ and $(f_2)$ imply that
	\begin{equation*}
	\int_{\R^{3}} f(u_{n}^{2})u_{n}^{2}dx\rightarrow 0.
	\end{equation*}
	Since $u_{n}\rightarrow 0$ in $L^{\frac{12}{4+3t}}(\R^{3}, \R)$, from Lemma \ref{poisson}-$(4)$ we deduce that
	$$
	\int_{\R^{3}} \phi^{t}_{u_{n}} u_{n}^{2}dx\rightarrow 0.
	$$
	Therefore
\begin{align*}
o_{n}(1)=\langle I'_{\mu}(u_{n}), u_{n}\rangle=\|u_{n}\|_{\mu}^{2}+\int_{\R^{3}} \phi^{t}_{u_{n}} u_{n}^{2}dx-\int_{\R^{3}} f(u_{n}^{2})u_{n}^{2}dx=\|u_{n}\|^{2}_{\mu}+o_{n}(1)
\end{align*}
which implies that $u_{n}\rightarrow 0$ in $H^{s}_{\mu}$ as $n\rightarrow \infty$.
\end{proof}

\noindent
In the next result we show that $m_{\mu}$ can be achieved.
\begin{lem}\label{FS}
Let $(u_{n})\subset \mathcal{M}_{\mu}$ be a sequence satisfying $I_{\mu}(u_{n})\rightarrow m_{\mu}$. Then, up to subsequences, the following alternatives hold:
\begin{compactenum}[(i)]
\item $(u_{n})$ strongly converges in $H^{s}_{\mu}$, 
\item there exists a sequence $(\tilde{y}_{n})\subset \R^{3}$ such that,  up to a subsequence, $v_{n}(x)=u_{n}(x+\tilde{y}_{n})$ converges strongly in $H^{s}_{\mu}$.
\end{compactenum}
In particular, there exists a minimizer $w\in H^{s}_{\mu}$ for $I_{\mu}$ with $I_{\mu}(w)=m_{\mu}$.
%\begin{lem}
%For any $\mu>0$, problem \eqref{AP} has a ground state solution. 
\end{lem}
\begin{proof}
Using a version of the mountain pass theorem without $(PS)$ condition (see \cite{W}), we may suppose that  $(u_{n})$ is a $(PS)_{m_{\mu}}$ sequence for $I_{\mu}$. Arguing as in Lemma \ref{bPS}, it is easy to check that $(u_{n})$ is bounded in $H^{s}_{\mu}$ so we may assume that $u_{n}\rightharpoonup u$ in $H^{s}_{\mu}$. The weak convergence together with Lemma \ref{Sembedding}, Lemma \ref{propf} and Lemma \ref{poisson} imply that $I'_{\mu}(u)=0$. 
Now, we assume that $u\neq 0$. Since $u\in \mathcal{M}_{\mu}$, we can use $(f_3)$ and Fatou's Lemma to see that
\begin{align*}
m_{\mu}&\leq I_{\mu}(u)-\frac{1}{4}\langle I'_{\mu}(u), u\rangle \\
&=\frac{1}{4}\|u\|_{\mu}^{2}+\frac{1}{2}\int_{\R^{3}}\frac{1}{2} f(u^{2})u-F(u^{2})\, dx \\
&\leq \liminf_{n\rightarrow \infty} \left[I_{\mu}(u_{n})-\frac{1}{4}\langle I'_{\mu}(u), u\rangle\right]=m_{\mu},
\end{align*}
which implies that $I_{\mu}(u)=m_{\mu}$. 

Let us consider the case $u=0$. Since $m_{\mu}>0$ and $I_{\mu}$ is continuous, we can see that $\|u_{n}\|_{\mu}\not\rightarrow 0$. Then we can use Lemma \ref{compactness} to find a sequence $(y_{n})\subset \R^{3}$ and constants $R, \beta>0$ such that 
$$
\liminf_{n\rightarrow \infty} \int_{B_{R}(y_{n})}|u_{n}|^{4}dx\geq \beta>0.
$$
Let us define $v_{n}=u_{n}(\cdot+y_{n})$, and we note that $v_{n}$ has a nontrivial weak limit $v$ in $H^{s}_{\mu}$. It is clear that also $(v_{n})$ is a $(PS)_{m_{\mu}}$ sequence for $I_{\mu}$, and arguing as before we can deduce that $I_{\mu}(v)=m_{\mu}$. In conclusion, we have proved that for all $\mu>0$, problem \eqref{AP} admits a ground state solution.

Now, let $u$ be a ground state for \eqref{AP}. Taking $\varphi=u^{-}$ as test function in $\langle I'_{\mu}(u), \varphi\rangle=0$, it is easy to check that $u\geq 0$ in $\R^{3}$. In particular, observing that $\phi_{u}^{t}\geq 0$ and $f$ has a subcritical growth, we can argue as in \cite[Proposition 5.1.1]{DMV} to see that $u\in L^{\infty}(\R^{3}, \R)$. 
In particular, we have
\begin{align*}
\phi_{u}^{t}(x)&= \int_{|y-x|\geq 1} \frac{|u(y)|^{2}}{|x-y|^{3-2t}}dy+ \int_{|y-x|<1} \frac{|u(y)|^{2}}{|x-y|^{3-2t}}dy\\
&\leq \|u\|_{L^{2}(\R^{3})}^{2}+\|u\|^{2}_{L^{\infty}(\R^{3})} \int_{|y-x|<1} \frac{1}{|x-y|^{3-2t}}dy\leq C,
\end{align*}
so that $g(x)=f(u^{2})u-\mu u-\phi_{u}^{t}u\in L^{\infty}(\R^{3}, \R)$. Applying \cite[Lemma $3.4$]{FQT} we can deduce that $u\in C^{0, \alpha}(\R^{3}, \R)$. Let $w$ be a solution to $-\Delta w=\mu u-\phi_{u}^{t}u+f(u^2)u\in C^{0, \alpha}(\R^{3}, \R)$. From the Schauder estimates for the Laplacian, we know that $w\in C^{2, \alpha}(\R^{3})$. It follows from $2s+\alpha>1$ that $(-\Delta)^{1-s}w\in C^{1, 2s+\alpha-1}$, and being $(-\Delta)^{s}(u-(-\Delta)^{1-s}w)=0$, we get that $u-(-\Delta)^{1-s}w$ is harmonic and $u$ has the same regularity of $(-\Delta)^{1-s}w$. Therefore $u\in C^{1, 2s+\alpha-1}(\R^{3}, \R)$. Recalling the following integral representation for the fractional Laplacian \cite[Lemma $3.2$]{DPV}
$$
(-\Delta)^{s}u(x)=-\frac{c_{3,s}}{2}\int_{\R^{3}} \frac{u(x+y)+u(x-y)-2u(x)}{|y|^{3+2s}} dy,
$$
we can see that if $u(x_{0})=0$ for some $x_{0}\in \R^{3}$, then 
$$
0>(-\Delta)^{s}u(x_{0})=-\mu u(x_{0})-\phi_{u}^{t}u(x_{0})+f(u(x_{0})^{2})u(x_{0})=0
$$
that is a contradiction.
Therefore $u>0$ in $\R^{3}$. Since $u\in C^{1, \gamma}(\R^{3}, \R)\cap L^{2}(\R^{3}, \R)$, we can deduce that $u(x)\rightarrow 0$ as $|x|\rightarrow \infty$. Then we can find $R>0$ such that $(-\Delta)^{s}u+\frac{\mu}{2}u\leq 0$ in $|x|>R$. Using \cite[Lemma 4.3]{FQT} we know that there exists a positive function $w$ such that for $|x|>R$ (taking $R$ larger if it is necessary), it holds $(-\Delta)^{s}w+\frac{\mu}{2}w\geq 0$ and $w(x)=\frac{C_{0}}{|x|^{3+2s}}$. 
In view of the continuity of $u$ and $w$ there exists some constant $C_{1}>0$ such that $z=u-C_{1}w\leq 0$ on $|x|=R$.
Moreover, we can see that $(-\Delta)^{s}z+\frac{\mu}{2}z\leq 0$ in $|x|\geq R$. From the maximum principle we can deduce that $z\leq 0$ in $|x|\geq R$, that is $0<u(x)\leq C_{1}w(x)\leq \frac{C_{2}}{|x|^{3+2s}}$ for all $|x|$ big enough. 
This last estimate will be useful to prove the existence of a nontrivial solution to \eqref{Pe}.
%Moreover, taking into account  the continuity of $u$ and $u(x)\rightarrow 0$ as $|x|\rightarrow \infty$, we can argue as in Lemma $4.3$ in \cite{FQT} to see that $u(x)$ behaves like $|x|^{-(3+2s)}$ as $|x|\rightarrow \infty$. 

\end{proof}

%\begin{lem}\label{lem4.3}
%Let $(w_{n})\subset \mathcal{M}_{\mu}$ be a sequence satisfying $I_{\mu}(w_{n})\rightarrow c_{\mu}$. Then $(w_{n})$ is bounded in $H^{s}(\R^{N},\R)$ and, up to a subsequence, $w_n\rightharpoonup w$ in $H^{s}(\R^{N},\R)$.  If $w\neq 0$, then $w_n\to w\in \mathcal{M}_{\mu}$ in $H^{s}(\R^{N},\R)$ and $w$ is a ground state for \eqref{Plim}.
%If $w=0$, then there exist $(\tilde{y}_{n})\subset \R^{N}$ and $\tilde{w}\in H^{s}(\R^{N},\R)\setminus\{0\}$ such that up to a subsequence $w_{n}(\cdot +\tilde{y}_{n})\to\tilde{w}\in \mathcal{M}_{\mu}$ in $H^{s}(\R^{N},\R)$ and $\tilde{w}$ is a ground state for \eqref{Plim}.
%\end{lem}

%\begin{remark}\label{remdecay}
%In view of \cite{teng} we can see that a ground state $\upsilon$ obtained in Lemma \ref{lem4.3} is H\"older continuous and has a power type decay at infinite, more precisely
%$$
%0<\upsilon(x)\leq \frac{C}{|x|^{N+2s}} \mbox{ if } |x|>1.
%$$
%\end{remark}

\section{A compactness condition}
In this section we prove some compactness results for the functional $J_{\e}$. 
We start proving the following property on the $(PS)_{d}$ sequences for $J_{\e}$ in the  noncoercive case $V_{\infty}<\infty$.
\begin{lem}\label{lem2.3}
Let $d\in\R$. Assume that $V_{\infty}<\infty$ and let $(v_{n})$ be a $(PS)_{d}$ sequence for $J_{\e}$ in $\h$ with $v_{n}\rightharpoonup 0$ in $\h$. If $v_{n}\not \rightarrow 0$ in $\h$, then $d\geq m_{V_{\infty}}$, where $m_{V_{\infty}}$ is the minimax level of $I_{V_{\infty}}$.
\end{lem}
\begin{proof}
	Let $(t_{n})\subset (0, +\infty)$ be a sequence such that $(t_{n}|v_{n}|)\subset \mathcal{M}_{V_{\infty}}$.
	Then our first aim is to prove that $\limsup_{n\rightarrow \infty} t_{n} \leq 1$.\\
	Assume by contradiction that there exist $\delta>0$ and a subsequence, still denoted by $(t_{n})$, such that 
	\begin{equation}\label{tv6}
	t_{n}\geq 1+ \delta \quad \forall n\in \mathbb{N}. 
	\end{equation}
	Since $(v_{n})$ is a $(PS)_{d}$ sequence for $J_{\e}$, we know that $(v_{n})$ is bounded in view of Lemma \ref{bPS} and from $\langle J'_{\e}(v_{n}), v_{n}\rangle=o_{n}(1)$ we can see that
	\begin{equation}\label{tv7}
	[v_{n}]_{A_{\e}}^{2} + \int_{\R^{3}} V_{\e} (x) |v_{n}|^{2} dx+\int_{\R^{3}} \phi_{|v_{n}|}^{t}|v_{n}|^{2}\, dx = \int_{\R^{3}} f(|v_{n}|^{2}) |v_{n}|^{2} dx +o_{n}(1). 
	\end{equation}
	Recalling that $t_{n}|v_{n}| \in \mathcal{M}_{V_{\infty}}$ we have
	\begin{equation}\label{tv8}
	 t_{n}^{2}([|v_{n}|]^{2}+  V_{\infty} |v_{n}|_2^{2})+ t_{n}^{4}\int_{\R^{3}} \phi_{|v_{n}|}^{t}|v_{n}|^{2}\, dx =  \int_{\R^{3}} f(t^{2}_{n}|v_{n}|^{2}) t_{n}^{2}|v_{n}|^{2} dx. 
	\end{equation}
	Putting together 
	%\eqref{tv6}, 
	\eqref{tv7}, \eqref{tv8} and applying
	%since }$f$ is {\color{blue}strictly }increasing by (\ref{f5}) and 
	Lemma \ref{DI} we obtain
	\begin{equation}\label{tv88}
	o_{n}(1)+\left(\frac{1}{t_{n}^{2}}-1\right)[v_{n}]_{A_{\e}}^{2}+\int_{\R^{3}} \left[ \frac{f(t^{2}_{n}|v_{n}|^{2})}{t_{n}^{2}|v_{n}|^{2}} - \frac{f(|v_{n}|^{2})}{|v_{n}|^{2}} \right]|v_{n}|^{4} \,dx\leq \int_{\R^{3}} \left( \frac{V_{\infty}}{t_{n}^{2}} - V_{\e} (x)\right) |v_{n}|^{2} dx +o_{n}(1).
	\end{equation}
Taking into account \eqref{RV}, we know that for every $\zeta>0$ there exists $R=R(\zeta)>0$ such that 
\begin{equation}\label{tv9}
\frac{V_{\infty}}{t_{n}^{2}} - V_{\e} (x) \leq  \zeta \quad \mbox{ for any } |x|\geq R. 
\end{equation}
In view of \eqref{tv9}, $|v_{n}|\rightarrow 0$ in $L^{2}(B_{R}(0), \R)$ (because Lemma \ref{Sembedding} and the weak convergence yield $v_{n}\rightarrow 0$ in $L^{2}(B_{R}(0), \C)$) and $(v_{n})$ in $\h$ is bounded, we have
%with the fact that,  by Lemma \ref{Sembedding}, $v_{n}\rightarrow 0$ in $L^{2}(B_{R}, \C)$, so that $|v_{n}|\rightarrow 0$ in $L^{2}(B_{R})$, and with the boundedness of $(v_{n})$ in $\h$, we get
	\begin{align*}
	\int_{\R^{3}} \left( \frac{V_{\infty}}{t_{n}^{2}} - V_{\e} (x)\right) |v_{n}|^{2} dx
	&\leq o_{n}(1)+ \zeta C. 
	\end{align*}
	This fact and \eqref{tv88} yield
	\begin{equation}\label{tv10}
	\int_{\R^{3}} \left[ \frac{f(t^{2}_{n}|v_{n}|^{2})}{t_{n}^{2}|v_{n}|^{2}} - \frac{f(|v_{n}|^{2})}{|v_{n}|^{2}} \right]|v_{n}|^{4} \,dx\leq \zeta C +o_{n}(1).
	\end{equation}
	Since $v_{n}\not \rightarrow 0$, we can use Lemma \ref{compactness} to find a sequence $(y_{n})\subset \R^{3}$, and two constants $\bar{R}, \beta$ such that
	\begin{equation}\label{tv11}
	\int_{B_{\bar{R}}(y_{n})} |v_{n}|^{4} dx \geq \beta>0.
	\end{equation}
	Set $w_{n}= |v_{n}|(\cdot+y_{n})$. By \eqref{RV}, Lemma \ref{DI} and the boundedness of $(v_{n})$ in $\h$, we deduce that $(w_{n})$ is bounded in $H^{s}(\R^{3}, \R)$, that is
	\[
	\|w_{n}\|_{V_{0}}^{2}= \|v_{n}\|_{V_{0}}^{2} \leq 
	%{\color{red}[v_{n}]^{2}_{A_{\e}}+ \int_{\R^{N}} V(\e x) |v_{n}|^{2} dx =} 
	\|v_{n}\|_{\e}^{2}\leq C.
	\]
	Hence $w_{n}\rightharpoonup w$ in $H^{s}(\R^{3}, \R)$ and $w_{n}\rightarrow w$ in $L^{4}_{\rm loc}(\R^{3}, \R)$. Moreover, by \eqref{tv11}, there exists $\Omega \subset \R^{3}$ with positive measure and such that $w\neq 0$ in $\Omega$.
	Putting together \eqref{tv6}  and \eqref{tv10} we can infer
	\begin{align*}
	0<\int_{\Omega}   \left[ \frac{f((1+\delta)^{2}w_{n}^{2})}{(1+\delta)^{2}w_{n}^{2}} - \frac{f(w_{n}^{2})}{w_{n}^{2}} \right]w_{n}^{4} \,dx    \leq \zeta C+o_{n}(1). 
	\end{align*}
	Taking the limit as $n\rightarrow \infty$ in the above inequality and applying Fatou's Lemma and $(f_4)$ we obtain
	\begin{align*}
	0<\int_{\Omega}  \left[ \frac{f((1+\delta)^{2}w^{2})}{(1+\delta)^{2}w^{2}} - \frac{f(w^{2})}{w^{2}} \right]w^{4} \,dx   \leq \zeta C 
	\end{align*}
	for any $\zeta>0$, which leads to a contradiction. \\
	Now, we consider the following two cases. \\
	{\bf Case 1:} $\limsup_{n\rightarrow \infty} t_{n}=1$.\\
	In this case there exists  a subsequence still denoted by $(t_{n})$ such that $t_{n}\rightarrow 1$. Since $(v_{n})$ is a $(PS)_{d}$ sequence for $J_{\e}$, $m_{V_{\infty}}$ is the minimax level of $I_{V_{\infty}}$, and Lemma \ref{DI}, we get
	\begin{equation}\label{tv12new}
	\begin{split}
	d+ o_{n}(1)&= J_{\e}(v_{n}) \\
	%&{\color{red}=J_{\e}(v_{n}) - I_{V_{\infty}}(t_{n}|v_{n}|)+ I_{V_{\infty}}(t_{n}|v_{n}|) } \\
	&\geq J_{\e}(v_{n}) -I_{V_{\infty}}(t_{n}|v_{n}|) + m_{V_{\infty}}\\
	&\geq \frac{1-t_{n}^{2}}{2} [|v_{n}|]^{2} + \frac{1}{2} \int_{\R^{3}} \left( V_{\e} (x) - t_{n}^{2} V_{\infty}\right) |v_{n}|^{2} dx \\
	&\qquad
	+\frac{1}{4}\int_{\R^{3}}(1-t_{n}^{4})\phi^{t}_{|v_{n}|}|v_{n}|^{2}\, dx +\frac{1}{2}\int_{\R^{3}} \left[ F(t^{2}_{n} |v_{n}|^{2}) -F(|v_{n}|^{2}) \right] \, dx+ m_{V_{\infty}}.
	\end{split}
	\end{equation}
%	Let us note that  in view of  we get
%	\begin{align}\begin{split}\label{tv12}
%	J_{\e}(v_{n}) &-I_{V_{\infty}}(t_{n}|v_{n}|) \\
%	&\geq \frac{1-t_{n}^{2}}{2} [|v_{n}|]^{2} + \frac{1}{2} \int_{\R^{N}} \left( V(\e x) - t_{n}^{2} V_{\infty}\right) |v_{n}|^{2} dx + \frac{1}{2}\int_{\R^{N}} \left[ F(t^{2}_{n} |v_{n}|^{2}) -F(|v_{n}|^{2}) \right] \, dx. 
%	\end{split} \end{align}
From the boundedness of $(|v_{n}|)$ in $H^{s}(\R^{3}, \R)$ and $t_{n}\rightarrow 1$, we have
\begin{equation}\label{tv14}
\frac{(1-t_{n}^{2})}{2} [|v_{n}|]^{2}= o_{n}(1). 
\end{equation}
	Using the Mean Value Theorem, Lemma \ref{propf}-$(i)$, $t_{n}\rightarrow 1$, and the boundedness of $(|v_{n}|)$, we can see that
	\begin{equation}\label{tv16}
	\begin{split}
	\int_{\R^{3}} \left[ F(t^{2}_{n} |v_{n}|^{2}) -F(|v_{n}|^{2}) \right] \, dx
	%&\leq
	%\int_{\R^{N}} |f(\theta_n |v_n|^2)| |t^{2}_{n}-1| |v_{n}|^{2} \, dx\\
	%&\leq
	%(C_1 |v_{n}|_2^{2}+ C_2 |v_{n}|_q^{q})  |t^{2}_{n}-1| 
	=o_{n}(1).
	\end{split}
	\end{equation}
	Then, \eqref{RV}, \eqref{tv16}, $(v_{n})$ is bounded in $\h$ yield
	\begin{align*}
	d+ o_{n}(1)\geq o_{n}(1) - \zeta C + m_{V_{\infty}}, 
	\end{align*}
	and taking the limit as $n\rightarrow \infty$ we can find $d \geq m_{V_{\infty}}$. \\
	{\bf Case 2:} $\limsup_{n\rightarrow \infty} t_{n}=t_{0}<1$. \\
	In this case there exists a subsequence still denoted by $(t_{n})$, such that $t_{n}\rightarrow t_{0}$ and $t_{n}<1$ for any $n\in \mathbb{N}$. 
	By \eqref{tv9}, $|v_{n}|\rightarrow 0$ in $L^{2}(B_{R}(0), \R)$ and $(v_{n})$ is  bounded, we can see that
	\begin{align}\label{tv17}
	\int_{\R^{3}} (V_{\infty}-V_{\e}(x))|v_{n}|^{2}\, dx\leq \zeta C+o_{n}(1). 
	\end{align}
	Let us note that the map $t\mapsto \frac{1}{2}f(t)t- F(t)$ is increasing for $t>0$ in view of $(f_4)$.
	This combined with $t_{n}|v_{n}|\in \mathcal{M}_{V_{\infty}}$, $t_{n}<1$, \eqref{tv17} and Lemma \ref{DI}, yields
	\begin{align*}
	m_{V_{\infty}} 
	%&\leq I_{V_{\infty}}(t_{n}|v_{n}|)  \\
	&\leq I_{V_{\infty}}(t_{n}|v_{n}|) -\frac{1}{4} \langle I'_{V_{\infty}}(t_{n}|v_{n}|),t_{n} |v_{n}|\rangle \\
	&=\frac{t_{n}^{2}}{4}\left([|v_{n}|]^{2}+V_{\infty}\|v_{n}\|_{L^{2}(\R^{3})}\right)+ \frac{1}{2} \int_{\R^{3}}  \left( \frac{1}{2} f(t^{2}_{n}|v_{n}|^{2}) t^{2}_{n}|v_{n}|^{2}- F(t^{2}_{n}|v_{n}|^{2}) \right) dx \\
	&\leq \frac{1}{4}\|v_{n}\|_{\e}^{2}+ \frac{1}{2} \int_{\R^{3}} \left( \frac{1}{2}f(|v_{n}|^{2}) |v_{n}|^{2} - F(|v_{n}|^{2})\right) \,dx+\zeta C+o_{n}(1) \\
	&=J_{\e}(v_{n})-\frac{1}{4}\langle J_{\e}'(v_{n}), v_{n}\rangle+\zeta C+o_{n}(1)\\
	&= d +\zeta C+o_{n}(1). 
	\end{align*}
	Letting the limit as $\zeta\rightarrow 0$ and then $n\rightarrow \infty$, we get $d\geq m_{V_{\infty}}$.
\end{proof}

\noindent
Now, we give the conditions on the levels $c$ for which $J_{\e}$ satisfies the $(PS)_{c}$ condition.
\begin{prop}\label{prop2.1}
	The functional $J_{\e}$ satisfies the $(PS)_{c}$ condition at any level $c<m_{V_{\infty}}$ if $V_{\infty}<\infty$, and at any level $c\in \R$ if $V_{\infty}=\infty$. 
\end{prop}

\begin{proof}
	Let $(u_{n})$ be a $(PS)_{c}$ sequence for $J_{\e}$. Hence $(u_{n})$ is bounded in $\h$ (see Lemma \ref{bPS}) and, up to a subsequence, we may assume that $u_{n}\rightharpoonup u$ in $\h$ and $u_{n}\rightarrow u$ in $L^{q}_{\rm loc}(\R^{3}, \C)$ for any $q\in [1, 2^{*}_{s})$.
	From assumptions $(f_1)$, $(f_2)$ and Lemma \ref{density}, Lemma \ref{Sembedding} and Lemma \ref{poisson}
	% and the fact that $C^{\infty}_{c}(\R^{N}, \C)$ is dense in $\h$ , 
	it is easy to see that $J'_{\e}(u)=0$. Moreover, by $(f_3)$, we can see that
	\begin{equation}\label{tv173}
	J_{\e}(u)=J_{\e}(u)-\frac{1}{4} \langle J'_{\e}(u),u\rangle=\frac{1}{4}\|u\|_{\e}^{2}+\frac{1}{2} \int_{\R^{3}} \left(\frac{1}{2}f(|u|^{2})|u|^{2}- F(|u|^{2})\right)dx\geq 0.
	\end{equation}
	Invoking Lemma \ref{vanishing} we can find a subsequence $(u_{n_{j}})\subset \h$ verifying \eqref{DL}.\\
	Now, let $v_{j}= u_{n_{j}}-\hat{u}_{j}$ where $\hat{u}_{j}$ is defined as in \eqref{eqtruncation}. 
	Using $(iii)$-$(iv)$ in Lemma \ref{propf} and $(i)$-$(ii)$ in Lemma \ref{CCS}, we can see that
	\begin{align}\label{tv171}
	J_{\e}(v_{j})= c-J_{\e}(u)+o_{j}(1)
	\end{align}
	and 
	\begin{align}\label{tv172}
	J'_{\e}(v_{j})=o_{j}(1).
	\end{align}
	Let us suppose that $V_{\infty}<\infty$ and $c<m_{V_\infty}$. From \eqref{tv173} we get $c-J_{\e}(u)\leq c <m_{V_\infty}$.
		Then, recalling that $(v_j)$ is a $(PS)_{c-J_{\e}(u)}$ sequence for $J_{\e}$ (by \eqref{tv171} and \eqref{tv172}) and that $v_{j}\rightharpoonup 0$ in $\h$, we can use  Lemma \ref{lem2.3} to deduce that $v_{j}\rightarrow 0$ in $\h$. Applying Lemma \ref{truncation} we can deduce that  $u_{n_{j}}\rightarrow u$ in $\h$ as $j\rightarrow \infty$.\\
	If $V_{\infty}=+\infty$ holds, we can use Lemma \ref{Sembedding},  $v_{j}\rightarrow 0$ in $L^{r}(\R^{3}, \C)$ for any $r\in [2, 2^{*}_{s})$, \eqref{tv172} and Lemma \ref{propf}-$(i)$ to infer that 
	$$
	\|v_{j}\|^{2}_{\e}+\int_{\R^{3}} \phi^{t}_{|v_{j}|}|v_{j}|^{2}\, dx=\int_{\R^{3}} f(|v_{j}|^{2})|v_{j}|^{2}dx+ o_{j}(1)=o_{j}(1).
	$$
	As before, we can deduce that $u_{n_{j}}\rightarrow u$ in $\h$ as $j\rightarrow \infty$ and this ends the proof of proposition.
\end{proof}

Now we show that $\N_{\e}$ is a natural constraint, namely that the constrained critical points of the functional $J_{\e}$ on $\N_{\e}$ are critical points of $J_{\e}$ in $\h$.
\begin{prop}\label{prop2.2}
	The functional $J_{\e}$ restricted to $\N_{\e}$ satisfies the $(PS)_{c}$ condition at any level $c<m_{V_{\infty}}$ if $V_{\infty}<\infty$, and at any level $c\in \R$ if $V_{\infty}=\infty$.
\end{prop}
\begin{proof}
	Let $(u_{n})\subset \N_{\e}$ be a $(PS)_{c}$ sequence of $J_{\e}$ restricted to $\N_{\e}$. Using \cite[Proposition 5.12]{W}, we can find a sequence $(\lambda_{n})\subset \R$ such that	
	\begin{equation}\label{tv18}
	J'_{\e}(u_{n})=\lambda_{n} T'_{\e}(u_{n})+o_{n}(1)
	\end{equation}
	where $	T_{\e}: H^{s}_{\e}\rightarrow \R$ is defined as
	$$
	T_{\e}(u)=\|u\|^{2}_{\e}+\int_{\R^{3}}\phi_{|u|}^{t}|u|^{2}\, dx-\int_{\R^{3}} f(|u|^{2})|u|^{2} dx.
	$$
	In light of $u_{n}\in \mathcal{N}_{\e}$ and $(f_5)$ we can see that
	\begin{align*}
	\langle T'_{\e}(u_{n}),u_{n}\rangle&=2\|u_{n}\|^{2}_{\e}+4 \int_{\R^{3}}\phi_{|u_{n}|}^{t}|u_{n}|^{2}\, dx -2\int_{\R^{3}} f(|u_{n}|^{2})|u_{n}|^{2}dx-2\int_{\R^{3}} f'(|u_{n}|^{2})|u_{n}|^{4}dx \\
	&=-2\|u_{n}\|^{2}_{\e}+ 2\int_{\R^{3}} [f(|u_{n}|^{2})|u_{n}|^{2}- f'(|u_{n}|^{2})|u_{n}|^{4}] dx\leq -2C_{\sigma} \|u_{n}\|_{L^{\sigma}(\R^{3})}^{\sigma}< 0.
	\end{align*}
	Then, up to a subsequence, we may assume that $\langle T'_{\e}(u_{n}),u_{n}\rangle\rightarrow \ell\leq 0$.\\
	If $\ell=0$, then $$o_{n}(1)=|\langle T'_{\e}(u_{n}),u_{n}\rangle|\geq C \|u_{n}\|_{L^{\sigma}(\R^{3})}^{\sigma}$$ so we obtain that $u_{n}\rightarrow 0$ in $L^{\sigma}(\R^{3}, \C)$.
	Since $(u_{n})\subset \N_{\e}$ and $J_{\e}(u_{n})\rightarrow c$  as $n\to\infty$, we can argue as in Lemma \ref{bPS} to see that $(u_{n})$ is bounded in $\h$. Then, by interpolation, we also have $u_{n}\rightarrow 0$ in $L^{q}(\R^{3}, \C)$. Hence, using Lemma \ref{propf}-$(i)$, we have
	$$
	\|u_{n}\|^{2}_{\e}\leq \|u_{n}\|^{2}_{\e}+\int_{\R^{3}} \phi_{|u_{n}|}|u_{n}|^{2}\, dx \leq\int_{\R^{3}} f(|u_{n}|^{2})|u_{n}|^{2}dx=o_{n}(1),
	$$
	which implies that $u_{n}\rightarrow 0$ in $\h$. This is impossible in view of Lemma \ref{LemNeharyE}-$(i)$. Therefore $\ell<0$ and by \eqref{tv18} we deduce that $\lambda_{n}=o_{n}(1)$. 
	From \eqref{tv18}, we have $J'_{\e}(u_{n})=o_{n}(1)$, that is $(u_{n})$ is a $(PS)_{c}$ sequence for $J_{\e}$. Then we can apply Proposition \ref{prop2.1} to obtain the thesis.
\end{proof}

\noindent
Arguing as before we can see that the following result holds.
\begin{cor}\label{cor2.1}
The constrained critical points of the functional $J_{\e}$ on $\N_{\e}$ are critical points of $J_{\e}$ in $\h$.
\end{cor}

\section{An existence result for \eqref{Pe}}
In this section we give a first existence result to \eqref{Pe}. More precisely:
\begin{thm}\label{Exthm}
Assume that \eqref{RV} and $(f_1)$-$(f_5)$ hold. Then, there exists $\e_{0}>0$ such that, for any $\e\in (0, \e_{0})$, problem \eqref{Pe} has a nontrivial solution.
\end{thm}
\begin{proof}
 Since $J_{\e}$ has a mountain pass geometry (see  Lemma \ref{MPG}), we can apply the Ekeland Variational Principle, to find a $(PS)_{c_{\e}}$ sequence $(u_{n})\subset \h$ for $J_{\e}$.
	
	If $V_{\infty}=\infty$, by Lemma \ref{Sembedding} and Proposition \ref{prop2.1} we deduce that $J_{\e}(u)=c_{\e}$ and $J'_{\e}(u)=0$, where $u\in \h$ is the weak limit of $u_{n}$.\\
	Now, assume that $V_{\infty}<\infty$. 
	%In view of Proposition \ref{prop2.1} it is enough to show that $c_{\e}<m_{V_{\infty}}$. 
	Without loss of generality, we may assume that 
	$$
	V(0)=V_{0}=\inf_{x\in \R^{3}} V(x).
	$$
	Fix $\mu\in  (V_{0}, V_{\infty})$. Clearly $m_{V_{0}}<m_{\mu}<m_{V_{\infty}}$. Let $w\in H^{s}(\R^{3}, \R)$ be a positive ground state to autonomous problem \eqref{AP} (which there exists in view of Lemma \ref{FS}) and we recall that $w\in C^{1, \gamma}(\R^{3}, \R)\cap L^{\infty}(\R^{3}, \R)$ and $0<w(x)\leq \frac{C}{|x|^{3+2s}}$ for $|x|>1$. Let $\eta\in C^{\infty}_{c}(\R^{3}, \R)$ be a cut-off function such that $\eta=1$ in $B_{1}(0)$ and $\eta=0$ in $B_{2}^{c}(0)$. Let us define $w_{r}(x):=\eta_{r}(x)w(x) e^{\imath A(0)\cdot x}$, with $\eta_r(x)=\eta(x/r)$ for $r>0$, and we observe that $|w_{r}|=\eta_{r}w$ and $w_{r}\in \h$ in view of Lemma \ref{aux}.
	Take $t_{r}>0$ such that 
	\begin{equation*}
	I_{\mu}(t_{r} |w_{r}|)=\max_{t\geq 0} I_{\mu}(t |w_{r}|)
	\end{equation*}
	Let us prove that there exists $r$ sufficiently large such that $I_{\mu}(t_{r}|w_{r}|)<m_{V_{\infty}}$.\\
	Assume by contradiction that $I_{\mu}(t_{r}|w_{r}|)\geq m_{V_{\infty}}$ for any $r>0$. Using Lemma $5$ in \cite{PP} we can see that $|w_{r}|\rightarrow w$ in $H^{s}(\R^{3}, \R)$ as $r\rightarrow \infty$, and being $w\in \mathcal{M}_{\mu}$, we have $t_{r}\rightarrow 1$ and
	$$
	m_{V_{\infty}}\leq \liminf_{r\rightarrow \infty} I_{\mu}(t_{r}|w_{r}|)=I_{\mu}(w)=m_{\mu}
	$$
	which gives a contradiction because of $m_{V_{\infty}}>m_{\mu}$.
	Then we can find $r>0$ such that
	\begin{align}\label{tv19}
	I_{\mu}(t_{r}|w_{r}|)=\max_{\tau\geq 0} I_{\mu}(\tau (t_{r} |w_{r}|)) \mbox{ and } I_{\mu}(t_{r}|w_{r}|)<m_{V_{\infty}}.
	\end{align}
	%Arguing as in \cite{AD} we can prove that
	Now, we prove the following limit:
	\begin{equation}\label{limwr}
	\lim_{\e\rightarrow 0}[w_{r}]_{A_{\e}}^{2}=[\eta_{r}w]^{2}.
	\end{equation}
	Let us note that
	\begin{align*}
	[w_{r}]_{A_{\e}}^{2}
	&=\iint_{\R^{6}} \frac{|e^{\imath A(0)\cdot x}\eta_{r}(x)w(x)-e^{\imath A_{\e}(\frac{x+y}{2})\cdot (x-y)}e^{\imath A(0)\cdot y} \eta_{r}(y)w(y)|^{2}}{|x-y|^{3+2s}} dx dy \nonumber \\
	&=[\eta_{r} w]^{2}
	+\iint_{\R^{6}} \frac{\eta_{r}^2(y)w^2(y) |e^{\imath [A_{\e}(\frac{x+y}{2})-A(0)]\cdot (x-y)}-1|^{2}}{|x-y|^{3+2s}} dx dy\\
	&\quad+2\Re \iint_{\R^{6}} \frac{(\eta_{r}(x)w(x)-\eta_{r}(y)w(y))\eta_{r}(y)w(y)(1-e^{-\imath [A_{\e}(\frac{x+y}{2})-A(0)]\cdot (x-y)})}{|x-y|^{3+2s}} dx dy \\
	&=: [\eta_{r} w]^{2}+X_{\e}+2Y_{\e}.
	\end{align*}
	Since 
	$|Y_{\e}|\leq [\eta_{r} w] \sqrt{X_{\e}}$, it is enough to show that	$X_{\e}\rightarrow 0$ as $\e\rightarrow 0$ to infer that \eqref{limwr} holds.\\
	For $0<\beta<\alpha/({1+\alpha-s})$, we have
	\begin{equation}\label{Ye}
	\begin{split}
	X_{\e}
	&\leq \int_{\R^{3}} w^{2}(y) dy \int_{|x-y|\geq\e^{-\beta}} \frac{|e^{\imath [A_{\e}(\frac{x+y}{2})-A(0)]\cdot (x-y)}-1|^{2}}{|x-y|^{3+2s}} dx\\
	&+\int_{\R^{3}} w^{2}(y) dy  \int_{|x-y|<\e^{-\beta}} \frac{|e^{\imath [A_{\e}(\frac{x+y}{2})-A(0)]\cdot (x-y)}-1|^{2}}{|x-y|^{3+2s}} dx \\
	&=:X^{1}_{\e}+X^{2}_{\e}.
	\end{split}
	\end{equation}
	Since $|e^{\imath t}-1|^{2}\leq 4$ and recalling that $w\in H^{s}(\R^{3}, \R)$, we can observe that 
	\begin{equation}\label{Ye1}
	X_{\e}^{1}\leq C \int_{\R^{3}} w^{2}(y) dy \int_{\e^{-\beta}}^\infty \rho^{-1-2s} d\rho\leq C \e^{2\beta s} \rightarrow 0.
	\end{equation}
	Regarding $X^{2}_{\e}$, since $|e^{\imath t}-1|^{2}\leq t^{2}$ for all $t\in \R$, $A\in C^{0,\alpha}(\R^3,\R^3)$ for $\alpha\in(0,1]$, and $|x+y|^{2}\leq 2(|x-y|^{2}+4|y|^{2})$, we can obtain
	\begin{equation}\label{Ye2}
	\begin{split}
	X^{2}_{\e}&
	\leq \int_{\R^{3}} w^{2}(y) dy  \int_{|x-y|<\e^{-\beta}} \frac{|A_{\e}\left(\frac{x+y}{2}\right)-A(0)|^{2} }{|x-y|^{3+2s-2}} dx \\
	&\leq C\e^{2\alpha} \int_{\R^{3}} w^{2}(y) dy  \int_{|x-y|<\e^{-\beta}} \frac{|x+y|^{2\alpha} }{|x-y|^{3+2s-2}} dx \\
	&\leq C\e^{2\alpha} \left(\int_{\R^{3}} w^{2}(y) dy  \int_{|x-y|<\e^{-\beta}} \frac{1 }{|x-y|^{3+2s-2-2\alpha}} dx\right.\\
	&\qquad\qquad+ \left. \int_{\R^{3}} |y|^{2\alpha} w^{2}(y) dy  \int_{|x-y|<\e^{-\beta}} \frac{1}{|x-y|^{3+2s-2}} dx\right) \\
	&=: C\e^{2\alpha} (X^{2, 1}_{\e}+X^{2, 2}_{\e}).
	\end{split}
	\end{equation}	
	Therefore
	\begin{equation}\label{Ye21}
	X^{2, 1}_{\e}
	= C  \int_{\R^{3}} w^{2}(y) dy \int_0^{\e^{-\beta}} \rho^{1+2\alpha-2s} d\rho
	\leq C\e^{-2\beta(1+\alpha-s)}.
	\end{equation}
	On the other hand, recalling the polynomial decay estimate on $w$ we infer that
	\begin{equation}\label{Ye22}
	\begin{split}
	 X^{2, 2}_{\e}
	 &\leq C  \int_{\R^{3}} |y|^{2\alpha} w^{2}(y) dy \int_0^{\e^{-\beta}}\rho^{1-2s} d\rho  \\
	&\leq C \e^{-2\beta(1-s)} \left[\int_{B_1(0)}  w^{2}(y) dy + \int_{B_1^c(0)} \frac{1}{|y|^{2(3+2s)-2\alpha}} dy \right]  \\
	&\leq C \e^{-2\beta(1-s)}.
	\end{split}
	\end{equation}
	Taking into account \eqref{Ye}, \eqref{Ye1}, \eqref{Ye2}, \eqref{Ye21} and \eqref{Ye22} we can conclude that $X_{\e}\rightarrow 0$.
	%Finally, we also note that by using Lemma $2.3$-$(6)$ in \cite{teng}, we can see that as $\e\rightarrow 0$
	%$$
	%\int_{\R^{3}} \phi_{|w_{r}|}|w_{r}|^{2}\, dx\rightarrow \int_{\R^{3}} \phi_{|\eta_{r}w|}|\eta_{r}w|^{2}\, dx.
	%$$
	Now, in view of \eqref{RV}, there exists  $\e_{0}>0$ such that 
	\begin{equation}\label{tv20}
	V_{\e} (x)\leq \mu \mbox{ for all } x\in \supp(|w_{r}|), \e\in (0, \e_{0}).
	\end{equation} 
	Therefore, putting together \eqref{tv19} , \eqref{limwr} and \eqref{tv20}, 
	we deduce that
	$$
	\limsup_{\e\rightarrow 0}c_{\e}\leq \limsup_{\e\rightarrow 0}\left[\max_{\tau\geq 0} J_{\e}(\tau t_{r} w_{r})\right]\leq \max_{\tau\geq 0} I_{\mu}(\tau t_{r} |w_{r}|)=I_{\mu}(t_{r}|w_{r}|)<m_{V_{\infty}}
	$$ 
	which implies that $c_{\e}<m_{V_{\infty}}$ for any $\e>0$ sufficiently small.
	Then we can apply Proposition \ref{prop2.1} to deduce the thesis.
\end{proof}

\section{Multiple solutions to \eqref{Pe}}
%In this section, our main purpose is to apply the Ljusternik-Schnirelmann category theory to prove a multiplicity result for problem \eqref{Pe}. In order to achieve our main result, first we give some useful preliminary lemmas.\\
This section is devoted to the proof of a multiplicity result for \eqref{Pe}. For this purpose, we begin proving the following compactness result.
\begin{prop}\label{prop4.1}
	Let $\e_{n}\rightarrow 0^{+}$ and $(u_{n})\subset \N_{\e_{n}}$ be such that $J_{\e_{n}}(u_{n})\rightarrow m_{V_{0}}$. Then there exists $(\tilde{y}_{n})\subset \R^{3}$ such that the translated sequence 
	\begin{equation*}
	v_{n}(x):=|u_{n}|(x+ \tilde{y}_{n})
	\end{equation*}
	has a subsequence which converges in $H^{s}(\R^{3}, \R)$. Moreover, up to a subsequence, $(y_{n}):=(\e_{n}\tilde{y}_{n})$ is such that $y_{n}\rightarrow y\in M$. 
\end{prop}

\begin{proof}
	Since $\langle J'_{\e_{n}}(u_{n}), u_{n} \rangle=0$ and $J_{\e_{n}}(u_{n})\rightarrow m_{V_{0}}$, we can argue as in Lemma \ref{bPS} to see that $\|u_{n}\|_{\e_{n}}\leq C$ for all $n\in\mathbb{N}$.
	Let us note $\|u_{n}\|_{\e_{n}}\nrightarrow 0$ otherwise $J_{\e_{n}}(u_{n})\rightarrow 0$ which is impossible since $m_{V_{0}}>0$.
	Therefore, as in the proof of Lemma \ref{compactness}, we can find a sequence $(\tilde{y}_{n})\subset \R^{3}$ and constants $R, \beta>0$ such that
	\begin{equation}\label{tv21}
	\liminf_{n\rightarrow \infty}\int_{B_{R}(\tilde{y}_{n})} |u_{n}|^{4} dx\geq \beta.
	\end{equation}
	Let us define 
	\begin{equation*}
	v_{n}(x):=|u_{n}|(x+ \tilde{y}_{n}). 
	\end{equation*}
	Applying Lemma \ref{DI} we can see that $(|u_{n}|)$ is bounded in $H^{s}(\R^{3}, \R)$ and, using  \eqref{tv21}, we may suppose that $v_{n}\rightharpoonup v$ in $H^{s}(\R^{3}, \R)$ for some $v\neq 0$.\\
	Let $(t_{n})\subset (0, +\infty)$ be such that $w_{n}=t_{n} v_{n}\in \mathcal{M}_{V_{0}}$, and set $y_{n}:=\e_{n}\tilde{y}_{n}$.  Taking into account
	Lemma \ref{DI} and Lemma \ref{poisson} and that $u_{n}\in \mathcal{N}_{\e_{n}}$, we can see that
	\begin{align*}
	m_{V_{0}}\leq I_{V_{0}}(w_{n})\leq \max_{t\geq 0} J_{\e_{n}}(t u_{n})=J_{\e_{n}}(u_{n})=m_{V_{0}}+ o_{n}(1), 
	\end{align*}
	which yields $I_{V_{0}}(w_{n})\rightarrow m_{V_{0}}$. \\
	Now, using the fact that $(v_{n})$ and $(w_{n})$
	are bounded in $H^{s}(\R^{3}, \R)$ and $v_{n}\nrightarrow 0$ in $H^{s}(\R^{3}, \R)$, we can deduce that $(t_{n})$ is bounded. Therefore, up to a subsequence, we may assume that $t_{n}\rightarrow t_{0}\geq 0$.
	Let us show that $t_{0}>0$.
	Otherwise, if $t_{0}=0$, from the boundedness of $(v_{n})$, we get $w_{n}= t_{n}v_{n} \rightarrow 0$ in $H^{s}(\R^{3}, \R)$, that is $I_{V_{0}}(w_{n})\rightarrow 0$ which contradicts $m_{V_{0}}>0$.
	Thus, up to a subsequence, we may assume that $w_{n}\rightharpoonup w:= t_{0} v\neq 0$ in $H^{s}(\R^{3}, \R)$. 
	From Lemma \ref{FS}, we can deduce that $w_{n} \rightarrow w$ in $H^{s}(\R^{3}, \R)$, which gives $v_{n}\rightarrow v$ in $H^{s}(\R^{3}, \R)$. This concludes the first part of the proposition.\\
	Now, we aim to show that $(y_{n})$ has a bounded subsequence. 
	Assume by contradiction that there exists a subsequence, still denoted by $(y_{n})$, such that $|y_{n}|\rightarrow +\infty$. 
	Firstly, we consider the case $V_{\infty}=\infty$. 
	By Lemma \ref{DI}, we can note that
	$$
	\int_{\R^{3}} V(\e_{n}x+y_{n})|v_{n}|^{2} dx\leq [|v_{n}|]^{2}+\int_{\R^{3}} V(\e_{n}x+y_{n})|v_{n}|^{2} dx+\int_{\R^{3}} \phi_{|v_{n}|}^{t}|v_{n}|^{2}\, dx
	\leq\int_{\R^{3}} f(|v_{n}|^{2})|v_{n}|^{2}\,dx
	$$
	which together with Fatou's Lemma and \eqref{RV} implies that
	$$
	\infty=\liminf_{n\rightarrow \infty} \int_{\R^{3}} f(|v_{n}|^{2})|v_{n}|^{2}\,dx
	$$
	and this is impossible since $(f(|v_{n}|^{2})|v_{n}|^{2})$ is bounded in $L^{1}(\R^{3}, \R)$.
	
	Now, we assume that $V_{\infty}<\infty$. 
	Taking into account $w_{n}\rightarrow w$ in $H^{s}(\R^{3}, \R)$, $V_{0}<V_{\infty}$, Lemma \ref{DI} and Lemma \ref{poisson}-$(3)$, we can see that
	\begin{equation}\label{tv22}
	\begin{split}
	m_{V_{0}}&= I_{V_{0}}(w) < I_{V_{\infty}} (w)\\
	&\leq \liminf_{n\rightarrow \infty} \left[\frac{1}{2} [w_{n}]^{2}+\frac{1}{2} \int_{\R^{3}} V(\e_{n} x+y_{n})|w_{n}|^{2} dx+\frac{1}{4}\int_{\R^{3}} \phi^{t}_{|w_{n}|}|w_{n}|^{2}\, dx-\frac{1}{2}\int_{\R^{3}} F(|w_{n}|^{2}) dx\right] \\
	&=\liminf_{n\rightarrow \infty} \left[\frac{t^{2}_{n}}{2} [|u_{n}|]^{2}+\frac{t^{2}_{n}}{2} \int_{\R^{3}} V(\e_{n} z)|u_{n}|^{2} dx+\frac{t_{n}^{4}}{4}\int_{\R^{3}} \phi^{t}_{|u_{n}|}|u_{n}|^{2}\, dx-\frac{1}{2}\int_{\R^{3}} F(t_{n}^{2} |u_{n}|^{2}) dx\right] \\
	&\leq \liminf_{n\rightarrow \infty} J_{\e_{n}}(t_{n}u_{n}) \leq \liminf_{n\rightarrow \infty} J_{\e_{n}} (u_{n})=m_{V_{0}}
	\end{split}
	\end{equation}
	which gives a contradiction. \\
	Therefore, $(y_{n})$ is bounded and, up to a subsequence, we may assume that $y_{n}\rightarrow y$. If $y\notin M$, then $V_{0}<V(y)$ and we can argue as in \eqref{tv22} to deduce a contradiction. Thus $y\in M$ and this ends the proof of proposition.	
\end{proof}

Let $\delta>0$ be fixed and $\omega\in H^s(\R^3, \R)$ be a ground state solution of problem \eqref{AP} for $\mu=V_0$ given by Lemma \ref{FS}.
Let $\psi\in C^{\infty}(\R^{+}, [0, 1])$ be a nonincreasing function such that $\psi=1$ in $[0, \delta/2]$ and $\psi=0$ in $[\delta, \infty)$. 

For any $y\in M$, we define
$$
\Psi_{\e, y}(x):=\psi(|\e x-y|) \omega\left(\frac{\e x-y}{\e}\right)e^{\imath \tau_{y}(\frac{\e x-y}{\e})}
$$
where $M$ is defined in \eqref{defM} and $\tau_{y}(x):=\sum_{j=1}^{3} A_{j}(y)x_{j}$, and
let $t_{\e}>0$ be the unique number such that
$$
J_{\e}(t_{\e}\Psi_{\e, y})=\max_{t\geq 0} J_{\e}(t_{\e}\Psi_{\e, y}).
$$
Let us introduce the map $\Phi_{\e}:M\rightarrow \N_{\e}$ by setting $\Phi_{\e}(y)=t_{\e} \Psi_{\e, y}$. By construction, $\Phi_{\e}(y)$ has compact support for any $y\in M$.
%Let us recall the following result proved in \cite{AD} (see Lemma $4.1$ there).
%\begin{lem}\label{ADlem}
%As $\e\rightarrow 0$ we have that $\|\Psi_{\e_{n}, y_{n}}\|_{\e}^2 \to \|\omega\|_{V_0}^2$ uniformly with respect to $y\in M$. 
%\end{lem}
\begin{lem}\label{lem4.1}
	The functional $\Phi_{\e}$ satisfies the following limit
	\begin{equation*}
	\lim_{\e\rightarrow 0} J_{\e}(\Phi_{\e}(y))=m_{V_{0}} \mbox{ uniformly in } y\in M.
	\end{equation*}
\end{lem}
\begin{proof}
	Assume by contradiction that there there exist $\kappa>0$, $(y_{n})\subset M$ and $\e_{n}\rightarrow 0$ such that 
	\begin{equation*}
	|J_{\e_{n}}(\Phi_{\e_{n}}(y_{n}))-m_{V_{0}}|\geq \kappa.
	\end{equation*}
	Since $\langle J'_{\e_{n}}(\Phi_{\e_{n}}(y_{n})), \Phi_{\e_{n}}(y_{n})\rangle=0$ and using the change of variable $z=(\e_{n}x-y_{n})/{\e_{n}}$, and that, if $z\in B_{\delta/\e_{n}}(0)$ then $\e_{n} z+y_{n}\in B_{\delta}(y_{n})\subset M_{\delta}$, we can see that from
	\begin{align}\label{DPsi}
	&\|\Psi_{\e_{n}, y_{n}}\|^{2}_{\e_{n}}+t_{\e_{n}}^{2}\int_{\R^{3}} \phi^{t}_{|\Psi_{\e_{n}},y_{n}|}|\Psi_{\e_{n},y_{n}}|^{2}\, dx\nonumber\\
	&=\int_{\R^{3}} f(|t_{\e_{n}} \psi(|\e_{n}z|) \omega(z)|^{2}) |\psi(|\e_{n}z|) \omega(z)|^{2}  dz,
	\end{align}
	we have
	\begin{equation*}
	\begin{split}
	\frac{1}{t_{\e_{n}}^{2}}\|\Psi_{\e_{n}, y_{n}}\|^{2}_{\e_{n}}+\int_{\R^{3}} \phi^{t}_{|\Psi_{\e_{n}},y_{n}|}|\Psi_{\e_{n},y_{n}}|^{2}\, dx
	&=\frac{1}{t_{\e_{n}}^{2}}\int_{\R^{3}} f(|t_{\e_{n}}\Psi_{\e_{n},y_{n}}|^{2}) |\Psi_{\e_{n}, y_{n}}|^{2} dx\\
	&=\frac{1}{t_{\e_{n}}^{2}}\int_{\R^{3}} f(|t_{\e_{n}} \psi(|\e_{n}z|) \omega(z)|^{2}) |\psi(|\e_{n}z|) \omega(z)|^{2}  dz\\
	&\geq \frac{1}{t_{\e_{n}}^{2}}\int_{B_{{\delta}/{2}}(0)} f(|t_{\e_{n}} \omega(z)|^{2}) \omega^{2}(z)  dz\\
	&\geq \frac{f(|t_{n}\alpha|^{2})}{|t_{n}\alpha|^{2}}  \int_{B_{{\delta}/{2}}(0)}\omega^{4}(z)  dz
	\end{split}
	\end{equation*}
	for all $n\geq n_{0}$, with  $n_{0}\in \mathbb{N}$ such that $B_{\frac{\delta}{2}}(0)\subset B_{\frac{\delta}{2\e_{n}}}(0)$ and $\alpha=\min\{\omega(z): |z|\leq \frac{\delta}{2}\}$. In the last passage we used $(f_4)$.
	Now, using Lemma $4.1$ in \cite{AD}, we can note that as $n\rightarrow \infty$
	\begin{equation}\label{ADlem}
	\|\Psi_{\e_{n}, y_{n}}\|_{\e_{n}}^2 \to \|\omega\|_{V_0}^{2}.
	\end{equation}
	Then, if $t_{\e_{n}}\rightarrow \infty$, in view of $(f_3)$ and \eqref{ADlem} we get
	$$
	\int_{\R^{3}} \phi^{t}_{|\Psi_{\e_{n}},y_{n}|}|\Psi_{\e_{n},y_{n}}|^{2}\, dx\rightarrow \infty
	$$
	which is a contradiction because $|\Psi_{\e_{n},y_{n}}|=\psi(|\e x-y|) \omega\left(\frac{\e x-y}{\e}\right)$ converges strongly to $\omega$ in $H^{s}(\R^{3}, \R)$ (see \cite[Lemma 5]{PP}), and using property $(6)$ in \cite[Lemma $2.3$]{teng}, we can see that
	$$
	\int_{\R^{3}} \phi^{t}_{|\Psi_{\e_{n}},y_{n}|}|\Psi_{\e_{n},y_{n}}|^{2}\, dx\rightarrow \int_{\R^{3}} \phi^{t}_{|\omega|} |\omega|^{2}\, dx.
	$$
	Therefore, up to a subsequence, we may assume that $t_{\e_{n}}\rightarrow t_{0}\geq 0$.
	Putting together \eqref{DPsi}, $(i)$ in Lemma \ref{LemNeharyE}, \eqref{ADlem} and $(f_1)$-$(f_2)$ we can see that $t_{0}>0$.\\
	Now, taking the limit as $n\rightarrow \infty$ in \eqref{DPsi}, we get
	$$
	[\omega]^{2}+\int_{\R^{3}} V_{0}|\omega|^{2} dx+t_{0}^{2}\int_{\R^{3}}\phi^{t}_{|\omega|}|\omega|^{2}\, dx=\int_{\R^{3}} f(|t_{0}\omega|^{2}) \omega^{2},
	$$ 
	that is $t_{0}\omega\in \mathcal{M}_{V_0}$. On the other hand, recalling that $\omega\in  \mathcal{M}_{V_0}$, we can see that
	$$
	\left(\frac{1}{t_{0}^{2}}-1 \right)\|\omega\|^{2}_{V_{0}}=\int_{\R^{3}} \left(\frac{f(|t_{0}\omega|^{2})}{|t_{0}\omega|^{2}}-\frac{f(|\omega|^{2})}{|\omega|^{2}}\right)|\omega|^{4}\, dx.
	$$ 
	Using $(f_4)$, we get $t_{0}=1$.
	Thus, invoking the Dominated Convergence Theorem and $t_{\e_{n}}\rightarrow 1$, we can see that
	\begin{equation*}%\label{4.45}
	\lim_{n\rightarrow \infty}\int_{\R^{3}} F(|\Phi_{\e_{n}}(y_{n})|^{2})=\int_{\R^{3}} F(\omega^{2}).
	\end{equation*}
	and than we can deduce that
	$$
	\lim_{n\rightarrow \infty} J_{\e_{n}}(\Phi_{\e_{n}}(y_{n}))=I_{V_{0}}(\omega)=m_{V_{0}}
	$$
	which is impossible.
\end{proof}

Now, for any $\delta>0$,  let $\rho=\rho(\delta)>0$ be such that $M_{\delta}\subset B_{\rho}(0)$ and we define $\Upsilon: \R^{3}\rightarrow \R^{3}$ by setting
\begin{equation*}
\Upsilon(x)=
\left\{
\begin{array}{ll}
x &\mbox{ if } |x|<\rho \\
\rho x/{|x|} &\mbox{ if } |x|\geq \rho.
\end{array}
\right.
\end{equation*}
Finally, we consider the barycenter map $\beta_{\e}: \N_{\e}\rightarrow \R^{3}$ given by
$$
\beta_{\e}(u):=\frac{\displaystyle\int_{\R^{3}} \Upsilon(\e x) |u(x)|^{4} dx}{\displaystyle\int_{\R^{3}} |u(x)|^{4} dx}.
$$
\begin{lem}\label{lem4.2}
	The function $\Phi_{\e}$ verifies the following limit
	\begin{equation*}%\label{3.3}
	\lim_{\e \rightarrow 0} \beta_{\e}(\Phi_{\e}(y))=y \mbox{ uniformly in } y\in M.
	\end{equation*}
	\end{lem}
	\begin{proof}
Assume by contradiction that there exist $\kappa>0$, $(y_{n})\subset M$ and $\e_{n}\rightarrow 0$ such that 
	\begin{equation}\label{4.4}
	|\beta_{\e_{n}}(\Phi_{\e_{n}}(y_{n}))-y_{n}|\geq \kappa.
	\end{equation}
	Set $z= ({\e_{n} x-y_{n}})/{\e_{n}}$, and we have 
	$$
	\beta_{\e_{n}}(\Psi_{\e_{n}}(y_{n}))=y_{n}+\frac{\int_{\R^{3}}[\Upsilon(\e_{n}z+y_{n})-y_{n}] |\psi(|\e_{n}z|)|^{4} |\omega(z)|^{4} \, dz}{\int_{\R^{3}} |\psi(|\e_{n}z|)|^{2} |\omega(z)|^{4}\, dz}.
	$$
	Since $(y_{n})\subset M\subset M_\delta \subset B_{\rho}(0)$, it follows from the Dominated Convergence Theorem that 
	$$
	|\beta_{\e_{n}}(\Phi_{\e_{n}}(y_{n}))-y_{n}|=o_{n}(1)
	$$
	which is an absurd in view of (\ref{4.4}).
\end{proof}

\noindent
At this point, we introduce a subset $\widetilde{\N}_{\e}$ of $\N_{\e}$ by setting 
$$
\widetilde{\N}_{\e}=\{u\in \N_{\e}: J_{\e}(u)\leq m_{V_{0}}+h(\e)\},
$$
where $h:\R_{+}\rightarrow \R_{+}$ is such that $h(\varsigma)\rightarrow 0$ as $\varsigma\rightarrow 0$.\\
Fixed $y\in M$, we conclude from Lemma \ref{lem4.1} that $h(\varsigma)=|J_{\varsigma}(\Phi_{\varsigma}(y))-m_{V_{0}}|\rightarrow 0$ as $\varsigma \rightarrow 0$. Hence $\Phi_{\e}(y)\in \widetilde{\N}_{\e}$, and $\widetilde{\N}_{\e}\neq \emptyset$ for any $\e>0$.
Moreover, we have the following interesting relation between $\widetilde{\N}_{\e}$ and the barycenter map.
\begin{lem}\label{lem4.4}
For any $\delta>0$, it holds
	$$
	\lim_{\e \rightarrow 0} \sup_{u\in \widetilde{\mathcal{N}}_{\e}} \operatorname{dist}(\beta_{\e}(u), M_{\delta})=0.
	$$
\end{lem}
\begin{proof}
	Let $\e_{n}\rightarrow 0$ as $n\rightarrow \infty$. For any $n\in \mathbb{N}$, there exists $(u_{n})\in \widetilde{\N}_{\e_{n}}$ such that
	$$
	\sup_{u\in \widetilde{\N}_{\e_{n}}} \inf_{y\in M_{\delta}}|\beta_{\e_{n}}(u)-y|=\inf_{y\in M_{\delta}}|\beta_{\e_{n}}(u_{n})-y|+o_{n}(1).
	$$
	Then, it is enough to verify that there exists $(y_{n})\subset M_{\delta}$ such that 
	\begin{equation}\label{3.13}
	\lim_{n\rightarrow \infty} |\beta_{\e_{n}}(u_{n})-y_{n}|=0.
	\end{equation}
	By Lemma \ref{DI}, we can see that $I_{V_{0}}(t|u_{n}|)\leq J_{\e_{n}}(t u_{n})$ for any $t\geq 0$. This fact and recalling that $(u_{n})\subset  \widetilde{\N}_{\e_{n}}\subset  \N_{\e_{n}}$ yield
	$$
	m_{V_{0}}\leq \max_{t\geq 0} I_{V_{0}}(t|u_{n}|)\leq \max_{t\geq 0} J_{\e_{n}}(t u_{n})=J_{\e_{n}}(u_{n})\leq m_{V_{0}}+h(\e_{n})
	$$
	which gives $J_{\e_{n}}(u_{n})\rightarrow m_{V_{0}}$ being $h(\e_{n})\rightarrow 0$ as $n\rightarrow \infty$. \\
	Then, applying Proposition \ref{prop4.1}, we can see that there exists $(\tilde{y}_{n})\subset \R^{3}$ such that $y_{n}=\e_{n}\tilde{y}_{n}\in M_{\delta}$ for $n$ sufficiently large. \\
	Hence
	$$
	\beta_{\e_{n}}(u_{n})=y_{n}+\frac{\int_{\R^{3}}[\Upsilon(\e_{n}z+y_{n})-y_{n}] |u_{n}(z+\tilde{y}_{n})|^{4} \, dz}{\int_{\R^{3}} |u_{n}(z+\tilde{y}_{n})|^{4} \, dz}.
	$$
	Since, up to a subsequence, $|u_{n}|(\cdot+\tilde{y}_{n})$ converges strongly in $H^{s}(\R^{3}, \R)$ and $\e_{n}z+y_{n}\rightarrow y\in M$ for any $z\in\R^{3}$, we can infer that (\ref{3.13}) holds true.
\end{proof}

\noindent
Finally, we give the proof of our multiplicity result.
\begin{proof}[Proof of Theorem \ref{thm1}]
	Fix $\delta>0$. Using Lemma \ref{lem4.1}, Lemma \ref{lem4.2} and Lemma \ref{lem4.4} and arguing as in  \cite[Section $6$]{CL}, we can find $\e_{\delta}>0$ such that for any $\e\in (0, \e_{\delta})$, the diagram
	$$
	M \stackrel{\Phi_{\e}}{\rightarrow}  \widetilde{\N}_{\e} \stackrel{\beta_{\e}}{\rightarrow} M_{\delta}
	$$
	is well-defined and $\beta_{\e}\circ \Phi_{\e}$ is  homotopically equivalent to the embedding $\iota: M\rightarrow M_{\delta}$. 
	This fact and \cite[Lemma 4.3]{BC} imply that
	$$
	\operatorname{cat}_{\widetilde{\N}_{\e}}(\widetilde{\N}_{\e})\geq \operatorname{cat}_{M_{\delta}}(M).
	$$
	In view of the definition $\widetilde{\N}_{\e}$ and Proposition \ref{prop2.2}, we know that $J_{\e}$ verifies the Palais-Smale condition in $\widetilde{\N}_{\e}$ (taking $\e_{\delta}$ smaller if necessary), so we can use  standard Ljusternik-Schnirelmann  theory for $C^{1}$ functionals (see \cite[Theorem $5.20$]{W}) to 
	deduce that $J_{\e}$ restricted to $\N_{\e}$ has at least $\operatorname{cat}_{M_{\delta}}(M)$ critical points. 
	Consequently, by Corollary \ref{cor2.1}, we can see that $J_{\e}$ has at least $\operatorname{cat}_{M_{\delta}}(M)$ critical points in $\h$.
\end{proof}

\section{Concentration phenomenon as $\e\rightarrow 0$}
In this last section we study the behavior of maximum points of the modulus of nontrivial solutions to \eqref{P}.
In order to do this, we first prove the following result in which we combine a suitable Moser-type iteration \cite{Moser}  and an approximation argument inspired by the Kato's inequality \cite{Kato}. 
\begin{lem}\label{moser} 
Let $\e_{n}\rightarrow 0$ and $u_{n}\in \mathcal{N}_{\e_{n}}$ be a solution to \eqref{Pe}. 
Set $v_{n}=|u_{n}|(\cdot+\tilde{y}_{n})$. Then $v_{n}\in L^{\infty}(\R^{3},\R)$ and there exists $C>0$ such that 
$$
\|v_{n}\|_{L^{\infty}(\R^{3})}\leq C \mbox{ for all } n\in \mathbb{N},
$$
where $\tilde{y}_{n}$ is given by Lemma \ref{prop4.1}.
Moreover
$$
\lim_{|x|\rightarrow \infty} v_{n}(x)=0 \mbox{ uniformly in } n\in \mathbb{N}.
$$
\end{lem}
\begin{proof}
For any $L>0$, we denote by $u_{L,n}:=\min\{|u_{n}|, L\}\geq 0$ and we define $v_{L, n}=u_{L,n}^{2(\beta-1)}u_{n}$ and $w_{L,n}:=|u_{n}| u_{L,n}^{\beta-1}$,  with $\beta>1$ to be determined later.
Taking $v_{L, n}$ as test function in (\ref{Pe}), we get
\begin{align}\label{conto1FF}
&\Re\left(\iint_{\R^{6}} \frac{(u_{n}(x)-u_{n}(y)e^{\imath A(\frac{x+y}{2})\cdot (x-y)})}{|x-y|^{3+2s}} \overline{(u_{n}u_{L,n}^{2(\beta-1)}(x)-u_{n}u_{L,n}^{2(\beta-1)}(y)e^{\imath A(\frac{x+y}{2})\cdot (x-y)})} \, dx dy\right)   \nonumber \\
&=-\int_{\R^{3}}\phi_{|u_{n}|}^{t}|u_{n}|^{2}u_{L,n}^{2(\beta-1)}dx+\int_{\R^{3}} f(|u_{n}|^{2}) |u_{n}|^{2}u_{L,n}^{2(\beta-1)}  \,dx-\int_{\R^{3}} V_{\e_{n}} (x) |u_{n}|^{2} u_{L,n}^{2(\beta-1)} \, dx.
\end{align}
Now, we can see that
\begin{align*}
&\Re\left[(u_{n}(x)-u_{n}(y)e^{\imath A(\frac{x+y}{2})\cdot (x-y)})\overline{(u_{n}u_{L,n}^{2(\beta-1)}(x)-u_{n}u_{L,n}^{2(\beta-1)}(y)e^{\imath A(\frac{x+y}{2})\cdot (x-y)})}\right] \\
&=\Re\Bigl[|u_{n}(x)|^{2}u_{L,n}^{2(\beta-1)}(x)-u_{n}(x)\overline{u_{n}(y)} u_{L,n}^{2(\beta-1)}(y)e^{-\imath A(\frac{x+y}{2})\cdot (x-y)}-u_{n}(y)\overline{u_{n}(x)} u_{L,n}^{2(\beta-1)}(x) e^{\imath A(\frac{x+y}{2})\cdot (x-y)} \\
&+|u_{n}(y)|^{2}u_{L,n}^{2(\beta-1)}(y) \Bigr] \\
&\geq (|u_{n}(x)|^{2}u_{L,n}^{2(\beta-1)}(x)-|u_{n}(x)||u_{n}(y)|u_{L,n}^{2(\beta-1)}(y)-|u_{n}(y)||u_{n}(x)|u_{L,n}^{2(\beta-1)}(x)+|u_{n}(y)|^{2}u^{2(\beta-1)}_{L,n}(y) \\
&=(|u_{n}(x)|-|u_{n}(y)|)(|u_{n}(x)|u_{L,n}^{2(\beta-1)}(x)-|u_{n}(y)|u_{L,n}^{2(\beta-1)}(y)),
\end{align*}
which yields
\begin{align}\label{realeF}
&\Re\left(\iint_{\R^{6}} \frac{(u_{n}(x)-u_{n}(y)e^{\imath A(\frac{x+y}{2})\cdot (x-y)})}{|x-y|^{3+2s}} \overline{(u_{n}u_{L,n}^{2(\beta-1)}(x)-u_{n}u_{L,n}^{2(\beta-1)}(y)e^{\imath A(\frac{x+y}{2})\cdot (x-y)})} \, dx dy\right) \nonumber\\
&\geq \iint_{\R^{6}} \frac{(|u_{n}(x)|-|u_{n}(y)|)}{|x-y|^{3+2s}} (|u_{n}(x)|u_{L,n}^{2(\beta-1)}(x)-|u_{n}(y)|u_{L,n}^{2(\beta-1)}(y))\, dx dy.
\end{align}
For all $t\geq 0$, set
\begin{equation*}
\gamma(t)=\gamma_{L, \beta}(t)=t t_{L}^{2(\beta-1)}
\end{equation*}
where  $t_{L}=\min\{t, L\}$. 
Since $\gamma$ is an increasing function, we have
\begin{align*}
(a-b)(\gamma(a)- \gamma(b))\geq 0 \quad \mbox{ for any } a, b\in \R.
\end{align*}
Let us define the functions 
\begin{equation*}
\Lambda(t)=\frac{|t|^{2}}{2} \quad \mbox{ and } \quad \Gamma(t)=\int_{0}^{t} (\gamma'(\tau))^{\frac{1}{2}} d\tau. 
\end{equation*}
and we note that
\begin{equation}\label{Gg}
\Lambda'(a-b)(\gamma(a)-\gamma(b))\geq |\Gamma(a)-\Gamma(b)|^{2} \mbox{ for any } a, b\in\R. 
\end{equation}
Indeed, for any $a, b\in \R$ such that $a<b$, the Jensen inequality yields
\begin{align*}
\Lambda'(a-b)(\gamma(a)-\gamma(b))&=(a-b)\int_{b}^{a} \gamma'(t)dt\\
&=(a-b)\int_{b}^{a} (\Gamma'(t))^{2}dt \\
&\geq \left(\int_{b}^{a} \Gamma'(t) dt\right)^{2}\\
&=(\Gamma(a)-\Gamma(b))^{2}.
\end{align*}
Analogously, we can prove that $\Lambda'(a-b)(\gamma(a)-\gamma(b))\geq (\Gamma(b)-\Gamma(a))^{2}$ for all $a\geq b$, which implies that \eqref{Gg} holds true.
From \eqref{Gg} we can deduce that
\begin{align}\label{Gg1}
|\Gamma(|u_{n}(x)|)- \Gamma(|u_{n}(y)|)|^{2} \leq (|u_{n}(x)|- |u_{n}(y)|)((|u_{n}|u_{L,n}^{2(\beta-1)})(x)- (|u_{n}|u_{L,n}^{2(\beta-1)})(y)). 
\end{align}
Putting together \eqref{realeF} and \eqref{Gg1}, we can see that
\begin{align}\label{conto1FFF}
\Re\left(\iint_{\R^{6}} \frac{(u_{n}(x)-u_{n}(y)e^{\imath A(\frac{x+y}{2})\cdot (x-y)})}{|x-y|^{3+2s}} \overline{(u_{n}u_{L,n}^{2(\beta-1)}(x)-u_{n}u_{L,n}^{2(\beta-1)}(y)e^{\imath A(\frac{x+y}{2})\cdot (x-y)})} \, dx dy\right) 
\geq [\Gamma(|u_{n}|)]^{2}.
\end{align}
%\begin{align}\label{BMS}
%&[\Gamma(v_{n})]^{2}+\int_{\R^{N}} V_{0}|v_{n}|^{2}v_{L, n}^{2(\beta-1)} dx \nonumber \\
%&\leq \iint_{\R^{2N}} \frac{(v_{n}(x)- v_{n}(y))}{|x-y|^{N+2s}} ((v_{n}v_{L, n}^{2(\beta-1)})(x)-(v_{n} v_{L,n}^{2(\beta-1)})(y)) \,dx dy +\int_{\R^{N}} V(\e_{n}x+\e_{n}\tilde{y}_{n})|v_{n}|^{2}v_{L,n}^{2(\beta-1)} dx \nonumber\\
%&\leq \int_{\R^{N}} g(\e_{n}x+\e_{n}\tilde{y}_{n}, v^{2}_{n}) v^{2}_{n} v_{L,n}^{2(\beta-1)} dx.
%\end{align}
Observing that $\Gamma(|u_{n}|)\geq \frac{1}{\beta} |u_{n}| u_{L,n}^{\beta-1}$ and using the fractional Sobolev inequality $\mathcal{D}^{s,2}(\R^{3}, \R)\subset L^{\2}(\R^{3}, \R)$, we find
\begin{equation}\label{SS1}
[\Gamma(|u_{n}|)]^{2}\geq S_{*} \|\Gamma(|u_{n}|)\|^{2}_{L^{\2}(\R^{3})}\geq \left(\frac{1}{\beta}\right)^{2} S_{*}\||u_{n}| u_{L,n}^{\beta-1}\|^{2}_{L^{\2}(\R^{3})}.
\end{equation}
Taking into account \eqref{conto1FF}, \eqref{conto1FFF}, \eqref{SS1} and property $(4)$ of Lemma \ref{poisson}, we obtain
\begin{align}\label{BMS}
\left(\frac{1}{\beta}\right)^{2} S_{*}\||u_{n}| u_{L,n}^{\beta-1}\|^{2}_{L^{\2}(\R^{3})}+\int_{\R^{3}} V_{\e_{n}} (x)|u_{n}|^{2}u_{L,n}^{2(\beta-1)} dx\leq \int_{\R^{3}} f(|u_{n}|^{2}) |u_{n}|^{2} u_{L,n}^{2(\beta-1)} dx.
\end{align}
Now, by $(f_1)$ and $(f_2)$, we can see that for any $\xi>0$ there exists $C_{\xi}>0$ such that
\begin{equation}\label{SS2}
f(t^{2})t^{2}\leq \xi |t|^{2}+C_{\xi}|t|^{\2} \mbox{ for all } t\in \R.
\end{equation}
Then, fixed $\xi\in (0, V_{0})$ and using \eqref{BMS} and \eqref{SS2} we get 
\begin{equation}\label{simo1}
\|w_{L,n}\|_{L^{\2}(\R^{3})}^{2}\leq C \beta^{2} \int_{\R^{3}} |u_{n}|^{\2}u_{L,n}^{2(\beta-1)}.
\end{equation}
Take $\beta=\frac{\2}{2}$ and fix $R>0$. Since $0\leq u_{L,n}\leq |u_{n}|$ and applying H\"older inequality we have
\begin{align}\label{simo2}
\int_{\R^{3}} |u_{n}|^{\2}u_{L,n}^{2(\beta-1)}dx&=\int_{\R^{3}} |u_{n}|^{\2-2} |u_{n}|^{2} u_{L,n}^{\2-2}dx \nonumber\\
&=\int_{\R^{3}} |u_{n}|^{\2-2} (|u_{n}| u_{L,n}^{\frac{\2-2}{2}})^{2}dx \nonumber\\
&\leq \int_{\{|u_{n}|<R\}} R^{\2-2} |u_{n}|^{\2} dx+\int_{\{|u_{n}|>R\}} |u_{n}|^{\2-2} (|u_{n}| u_{L,n}^{\frac{\2-2}{2}})^{2}dx \nonumber\\
&\leq \int_{\{|u_{n}|<R\}} R^{\2-2} |u_{n}|^{\2} dx+\left(\int_{\{|u_{n}|>R\}} |u_{n}|^{\2} dx\right)^{\frac{\2-2}{\2}} \left(\int_{\R^{3}} (|u_{n}| u_{L,n}^{\frac{\2-2}{2}})^{\2}dx\right)^{\frac{2}{\2}}.
\end{align}
Since $(|u_{n}|)$ is bounded in $H^{s}(\R^{3}, \R)$, we have for  $R$ big enough
\begin{equation}\label{simo3}
\left(\int_{\{|u_{n}|>R\}} |u_{n}|^{\2} dx\right)^{\frac{\2-2}{\2}}\leq  \frac{1}{2\beta^{2}}.
\end{equation}
In view of \eqref{simo1}, \eqref{simo2} and \eqref{simo3} we get
\begin{equation*}
\left(\int_{\R^{3}} (|u_{n}| u_{L,n}^{\frac{\2-2}{2}})^{\2} \right)^{\frac{2}{\2}}\leq C\beta^{2} \int_{\R^{3}} R^{\2-2} |u_{n}|^{\2} dx<\infty
\end{equation*}
and letting the limit as $L\rightarrow \infty$ we obtain $|u_{n}|\in L^{\frac{(\2)^{2}}{2}}(\R^{3},\R)$.

Now, using $0\leq u_{L,n}\leq |u_{n}|$ and passing to the limit as $L\rightarrow \infty$ in \eqref{simo1} we obtain
\begin{equation*}
\|u_{n}\|_{L^{\beta\2}(\R^{3})}^{2\beta}\leq C \beta^{2} \int_{\R^{3}} |u_{n}|^{\2+2(\beta-1)},
\end{equation*}
which implies that
\begin{equation*}
\left(\int_{\R^{3}} |u_{n}|^{\beta\2} dx\right)^{\frac{1}{(\beta-1)\2}}\leq (C \beta)^{\frac{1}{\beta-1}} \left(\int_{\R^{3}} |u_{n}|^{\2+2(\beta-1)}\right)^{\frac{1}{2(\beta-1)}}.
\end{equation*}
For $m\geq 1$ we define $\beta_{m+1}$ inductively so that $\2+2(\beta_{m+1}-1)=\2 \beta_{m}$ and $\beta_{1}=\frac{\2}{2}$. \\
Hence
\begin{equation*}
\left(\int_{\R^{3}} |u_{n}|^{\beta_{m+1}\2} dx\right)^{\frac{1}{(\beta_{m+1}-1)\2}}\leq (C \beta_{m+1})^{\frac{1}{\beta_{m+1}-1}} \left(\int_{\R^{3}} |u_{n}|^{\2\beta_{m}}\right)^{\frac{1}{\2(\beta_{m}-1)}}.
\end{equation*}
Set
$$
D_{m}=\left(\int_{\R^{3}} |u_{n}|^{\2\beta_{m}}\right)^{\frac{1}{\2(\beta_{m}-1)}}.
$$
Using an iterative argument, we can see that there exists $C_{0}>0$ independent of $m$ such that 
$$
D_{m+1}\leq \prod_{k=1}^{m} (C \beta_{k+1})^{\frac{1}{\beta_{k+1}-1}}  D_{1}\leq C_{0} D_{1}.
$$
Taking the limit as $m\rightarrow \infty$ we can deduce that
\begin{equation}\label{UBu}
\|u_{n}\|_{L^{\infty}(\R^{3})}\leq C_{0}D_{1}=:K \mbox{ for all } n\in \mathbb{N}.
\end{equation}
Consequently, by interpolation, $(|u_{n}|)$ strongly converges in $L^{r}(\R^{3}, \R)$ for all $r\in (2, \infty)$. Using $(f_1)$-$(f_2)$, we can also see that $f(|u_{n}|^{2})|u_{n}|$ strongly converges  in the same Lebesgue spaces. 

Next we show that $|u_{n}|$ is a weak subsolution to 
\begin{equation}\label{Kato0}
\left\{
\begin{array}{ll}
(-\Delta)^{s}v+V_{0} v=f(v^{2})v &\mbox{ in } \R^{3} \\
v\geq 0 \quad \mbox{ in } \R^{3}.
\end{array}
\right.
\end{equation}
Take $\varphi\in C^{\infty}_{c}(\R^{3}, \R)$ such that $\varphi\geq 0$. Let $u_{\delta,n}=\sqrt{|u_{n}|^{2}+\delta^{2}}$ for $\delta>0$ and we use $\psi_{\delta, n}=\frac{u_{n}}{u_{\delta, n}}\varphi$ as test function in \eqref{Pe}. 
We are going to prove that $\psi_{\delta, n}\in H^{s}_{\e_{n}}$ for all $\delta>0$ and $n\in \mathbb{N}$. Clearly, $\int_{\R^{3}} V_{\e_{n}} (x) |\psi_{\delta,n}|^{2} dx\leq \int_{\supp(\varphi)} V_{\e_{n}} (x)\varphi^{2} dx<\infty$. 
%For simplicity, we write $u$, $u_{\delta}$ and $\psi_{\delta}$ instead of $u_{n}$, $u_{\delta, n}$ and $\psi_{\delta,n}$. 
Now, we note that
\begin{align*}
\psi_{\delta,n}(x)-\psi_{\delta,n}(y)e^{\imath A_{\e}(\frac{x+y}{2})\cdot (x-y)}&=\left(\frac{u_{n}(x)}{u_{\delta,n}(x)}\right)\varphi(x)-\left(\frac{u_{n}(y)}{u_{\delta,n}(y)}\right)\varphi(y)e^{\imath A_{\e}(\frac{x+y}{2})\cdot (x-y)}\\
&=\left[\left(\frac{u_{n}(x)}{u_{\delta,n}(x)}\right)-\left(\frac{u_{n}(y)}{u_{\delta,n}(x)}\right)e^{\imath A_{\e}(\frac{x+y}{2})\cdot (x-y)}\right]\varphi(x) \\
&+\left[\varphi(x)-\varphi(y)\right] \left(\frac{u_{n}(y)}{u_{\delta,n}(x)}\right) e^{\imath A_{\e}(\frac{x+y}{2})\cdot (x-y)} \\
&+\left(\frac{u_{n}(y)}{u_{\delta,n}(x)}-\frac{u_{n}(y)}{u_{\delta,n}(y)}\right)\varphi(y) e^{\imath A_{\e}(\frac{x+y}{2})\cdot (x-y)}
\end{align*}
so we can see that
\begin{align*}
&|\psi_{\delta,n}(x)-\psi_{\delta,n}(y)e^{\imath A_{\e}(\frac{x+y}{2})\cdot (x-y)}|^{2} \\
&\leq \frac{4}{\delta^{2}}|u_{n}(x)-u_{n}(y)e^{\imath A_{\e}(\frac{x+y}{2})\cdot (x-y)}|^{2}\|\varphi\|^{2}_{L^{\infty}(\R^{3})} +\frac{4}{\delta^{2}}|\varphi(x)-\varphi(y)|^{2} \||u_{n}|\|^{2}_{L^{\infty}(\R^{3})} \\
&+\frac{4}{\delta^{4}} \||u_{n}|\|^{2}_{L^{\infty}(\R^{3})} \|\varphi\|^{2}_{L^{\infty}(\R^{3})} |u_{\delta,n}(y)-u_{\delta,n}(x)|^{2} \\
&\leq \frac{4}{\delta^{2}}|u_{n}(x)-u_{n}(y)e^{\imath A_{\e}(\frac{x+y}{2})\cdot (x-y)}|^{2}\|\varphi\|^{2}_{L^{\infty}(\R^{3})} +\frac{4K^{2}}{\delta^{2}}|\varphi(x)-\varphi(y)|^{2} \\
&+\frac{4K^{2}}{\delta^{4}} \|\varphi\|^{2}_{L^{\infty}(\R^{3})} ||u_{n}(y)|-|u_{n}(x)||^{2} 
\end{align*}
where we used $|z+w+k|^{2}\leq 4(|z|^{2}+|w|^{2}+|k|^{2})$ for all $z,w,k\in \C$, $|e^{\imath t}|=1$ for all $t\in \R$, $u_{\delta,n}\geq \delta$, $|\frac{u_{n}}{u_{\delta,n}}|\leq 1$, \eqref{UBu} and $|\sqrt{|z|^{2}+\delta^{2}}-\sqrt{|w|^{2}+\delta^{2}}|\leq ||z|-|w||$ for all $z, w\in \C$.\\
Since $u_{n}\in H^{s}_{\e_{n}}$, $|u_{n}|\in H^{s}(\R^{3}, \R)$ (by Lemma \ref{DI}) and $\varphi\in C^{\infty}_{c}(\R^{3}, \R)$, we can conclude that $\psi_{\delta,n}\in H^{s}_{\e_{n}}$.
Therefore
\begin{align}\label{Kato1}
&\Re\left[\iint_{\R^{6}} \frac{(u_{n}(x)-u_{n}(y)e^{\imath A_{\e}(\frac{x+y}{2})\cdot (x-y)})}{|x-y|^{3+2s}} \left(\frac{\overline{u_{n}(x)}}{u_{\delta,n}(x)}\varphi(x)-\frac{\overline{u_{n}(y)}}{u_{\delta,n}(y)}\varphi(y)e^{-\imath A_{\e}(\frac{x+y}{2})\cdot (x-y)}  \right) dx dy\right] \nonumber\\
&+\int_{\R^{3}} V_{\e_{n}} (x)\frac{|u_{n}|^{2}}{u_{\delta,n}}\varphi dx+\int_{\R^{3}} \phi_{|u_{n}|}^{t} \frac{|u_{n}|^{2}}{u_{\delta,n}}\varphi dx=\int_{\R^{3}} f( |u_{n}|^{2})\frac{|u_{n}|^{2}}{u_{\delta,n}}\varphi dx.
\end{align}
Since$\Re(z)\leq |z|$ for all $z\in \C$ and  $|e^{\imath t}|=1$ for all $t\in \R$, we get
\begin{align}\label{alves1}
&\Re\left[(u_{n}(x)-u_{n}(y)e^{\imath A_{\e}(\frac{x+y}{2})\cdot (x-y)}) \left(\frac{\overline{u_{n}(x)}}{u_{\delta,n}(x)}\varphi(x)-\frac{\overline{u_{n}(y)}}{u_{\delta,n}(y)}\varphi(y)e^{-\imath A_{\e}(\frac{x+y}{2})\cdot (x-y)}  \right)\right] \nonumber\\
&=\Re\left[\frac{|u_{n}(x)|^{2}}{u_{\delta,n}(x)}\varphi(x)+\frac{|u_{n}(y)|^{2}}{u_{\delta,n}(y)}\varphi(y)-\frac{u_{n}(x)\overline{u_{n}(y)}}{u_{\delta,n}(y)}\varphi(y)e^{-\imath A_{\e}(\frac{x+y}{2})\cdot (x-y)} -\frac{u_{n}(y)\overline{u_{n}(x)}}{u_{\delta,n}(x)}\varphi(x)e^{\imath A_{\e}(\frac{x+y}{2})\cdot (x-y)}\right] \nonumber \\
&\geq \left[\frac{|u_{n}(x)|^{2}}{u_{\delta,n}(x)}\varphi(x)+\frac{|u_{n}(y)|^{2}}{u_{\delta,n}(y)}\varphi(y)-|u_{n}(x)|\frac{|u_{n}(y)|}{u_{\delta,n}(y)}\varphi(y)-|u_{n}(y)|\frac{|u_{n}(x)|}{u_{\delta,n}(x)}\varphi(x) \right].
\end{align}
Now, we can observe that
\begin{align}\label{alves2}
&\frac{|u_{n}(x)|^{2}}{u_{\delta,n}(x)}\varphi(x)+\frac{|u_{n}(y)|^{2}}{u_{\delta,n}(y)}\varphi(y)-|u_{n}(x)|\frac{|u_{n}(y)|}{u_{\delta,n}(y)}\varphi(y)-|u_{n}(y)|\frac{|u_{n}(x)|}{u_{\delta,n}(x)}\varphi(x) \nonumber\\
&=  \frac{|u_{n}(x)|}{u_{\delta,n}(x)}(|u_{n}(x)|-|u_{n}(y)|)\varphi(x)-\frac{|u_{n}(y)|}{u_{\delta,n}(y)}(|u_{n}(x)|-|u_{n}(y)|)\varphi(y) \nonumber\\
&=\left[\frac{|u_{n}(x)|}{u_{\delta,n}(x)}(|u_{n}(x)|-|u_{n}(y)|)\varphi(x)-\frac{|u_{n}(x)|}{u_{\delta,n}(x)}(|u_{n}(x)|-|u_{n}(y)|)\varphi(y)\right] \nonumber\\
&+\left(\frac{|u_{n}(x)|}{u_{\delta,n}(x)}-\frac{|u_{n}(y)|}{u_{\delta,n}(y)} \right) (|u_{n}(x)|-|u_{n}(y)|)\varphi(y) \nonumber\\
&=\frac{|u_{n}(x)|}{u_{\delta,n}(x)}(|u_{n}(x)|-|u_{n}(y)|)(\varphi(x)-\varphi(y)) +\left(\frac{|u_{n}(x)|}{u_{\delta,n}(x)}-\frac{|u_{n}(y)|}{u_{\delta,n}(y)} \right) (|u_{n}(x)|-|u_{n}(y)|)\varphi(y) \nonumber\\
&\geq \frac{|u_{n}(x)|}{u_{\delta,n}(x)}(|u_{n}(x)|-|u_{n}(y)|)(\varphi(x)-\varphi(y)) 
\end{align}
where in the last inequality we used 
$$
\left(\frac{|u_{n}(x)|}{u_{\delta,n}(x)}-\frac{|u_{n}(y)|}{u_{\delta,n}(y)} \right) (|u_{n}(x)|-|u_{n}(y)|)\varphi(y)\geq 0
$$
since
$$
h(t)=\frac{t}{\sqrt{t^{2}+\delta^{2}}} \mbox{ is increasing for } t\geq 0 \quad \mbox{ and } \quad \varphi\geq 0 \mbox{ in }\R^{3}.
$$
Then, observing that
$$
\frac{|\frac{|u_{n}(x)|}{u_{\delta,n}(x)}(|u_{n}(x)|-|u_{n}(y)|)(\varphi(x)-\varphi(y))|}{|x-y|^{3+2s}}\leq \frac{||u_{n}(x)|-|u_{n}(y)||}{|x-y|^{\frac{3+2s}{2}}} \frac{|\varphi(x)-\varphi(y)|}{|x-y|^{\frac{3+2s}{2}}}\in L^{1}(\R^{6}),
$$
and $\frac{|u_{n}(x)|}{u_{\delta,n}(x)}\rightarrow 1$ a.e. in $\R^{3}$ as $\delta\rightarrow 0$,
and using \eqref{alves1}, \eqref{alves2}, we can apply the Dominated Convergence Theorem to obtain
\begin{align}\label{Kato2}
&\limsup_{\delta\rightarrow 0} \Re\left[\iint_{\R^{6}} \frac{(u_{n}(x)-u_{n}(y)e^{\imath A_{\e}(\frac{x+y}{2})\cdot (x-y)})}{|x-y|^{3+2s}} \left(\frac{\overline{u_{n}(x)}}{u_{\delta,n}(x)}\varphi(x)-\frac{\overline{u_{n}(y)}}{u_{\delta,n}(y)}\varphi(y)e^{-\imath A_{\e}(\frac{x+y}{2})\cdot (x-y)}  \right) dx dy\right] \nonumber\\
&\geq \limsup_{\delta\rightarrow 0} \iint_{\R^{6}} \frac{|u_{n}(x)|}{u_{\delta,n}(x)}(|u_{n}(x)|-|u_{n}(y)|)(\varphi(x)-\varphi(y)) \frac{dx dy}{|x-y|^{3+2s}} \nonumber\\
&=\iint_{\R^{6}} \frac{(|u_{n}(x)|-|u_{n}(y)|)(\varphi(x)-\varphi(y))}{|x-y|^{3+2s}} dx dy.
\end{align}
From the Dominated Convergence Theorem again (we recall that $\frac{|u_{n}|^{2}}{u_{\delta, n}}\leq |u_{n}|$), Fatou's Lemma and $\varphi\in C^{\infty}_{c}(\R^{3}, \R)$, we can also deduce that
\begin{equation}\label{Kato3}
\lim_{\delta\rightarrow 0} \int_{\R^{3}} V_{\e_{n}} (x)\frac{|u_{n}|^{2}}{u_{\delta,n}}\varphi dx=\int_{\R^{3}} V_{\e_{n}} (x)|u_{n}|\varphi dx\geq \int_{\R^{3}} V_{0}|u_{n}|\varphi dx
\end{equation}
\begin{equation}\label{KatoP}
\liminf_{\delta\rightarrow 0} \int_{\R^{3}} \phi_{|u_{n}|}^{t} \frac{|u_{n}|^{2}}{u_{\delta,n}}\varphi dx\geq \int_{\R^{3}} \phi_{|u|}^{t} |u|\varphi dx\geq 0
\end{equation}
and
\begin{equation}\label{Kato4}
\lim_{\delta\rightarrow 0}  \int_{\R^{3}} f( |u_{n}|^{2})\frac{|u_{n}|^{2}}{u_{\delta,n}}\varphi dx=\int_{\R^{3}} f(|u_{n}|^{2}) |u_{n}|\varphi dx.
\end{equation}
Taking into account \eqref{Kato1}, \eqref{Kato2}, \eqref{Kato3}, \eqref{KatoP} and \eqref{Kato4} we can see that
\begin{align*}
\iint_{\R^{6}} \frac{(|u_{n}(x)|-|u_{n}(y)|)(\varphi(x)-\varphi(y))}{|x-y|^{3+2s}} dx dy+\int_{\R^{3}} V_{0}|u_{n}|\varphi dx\leq 
\int_{\R^{3}} f(|u_{n}|^{2}) |u_{n}|\varphi dx
\end{align*}
for any $\varphi\in C^{\infty}_{c}(\R^{3}, \R)$ such that $\varphi\geq 0$. Hence $|u_{n}|$ is a weak subsolution to \eqref{Kato0}.
%Then, using $(V_{1})$, it is clear that $v_{n}=|u_{n}|(\cdot+\tilde{y}_{n})$ solves 
Now, we note that $v_{n}=|u_{n}|(\cdot+\tilde{y}_{n})$ solves
\begin{equation}\label{Pkat}
(-\Delta)^{s} v_{n} + V_{0}v_{n}\leq f(v_{n}^{2})v_{n} \mbox{ in } \R^{3}. 
\end{equation}
Let $z_{n}\in H^{s}(\R^{3}, \R)$ be the unique solution to
\begin{equation}\label{US}
(-\Delta)^{s} z_{n} + V_{0}z_{n}=g_{n} \mbox{ in } \R^{3},
\end{equation}
where
$$
g_{n}:=f(v_{n}^{2})v_{n}\in L^{r}(\R^{3}, \R) \quad \forall r\in [2, \infty].
$$
Since \eqref{UBu} implies that $\|v_{n}\|_{L^{\infty}(\R^{3})}\leq C$ for all $n\in \mathbb{N}$, by interpolation we can see that $v_{n}\rightarrow v$ strongly converges in $L^{r}(\R^{3}, \R)$ for all $r\in (2, \infty)$, for some $v\in L^{r}(\R^{3}, \R)$. Using $(f_1)$-$(f_2)$, we also have $g_{n}\rightarrow  f(v^{2})v$ in $L^{r}(\R^{3}, \R)$ and $\|g_{n}\|_{L^{\infty}(\R^{3})}\leq C$ for all $n\in \mathbb{N}$.
Now, being $z_{n}=\mathcal{K}*g_{n}$, where $\mathcal{K}$ is the Bessel kernel (see \cite[Section 3]{FQT}), and arguing as in \cite[Lemma 2.6]{AM}, we can see that $|z_{n}(x)|\rightarrow 0$ as $|x|\rightarrow \infty$ uniformly with respect to $n\in \mathbb{N}$.
Recalling that $v_{n}$ verifies \eqref{Pkat} and $z_{n}$ satisfies \eqref{US}, it is easy to use a comparison argument to deduce that $0\leq v_{n}\leq z_{n}$ a.e. in $\R^{3}$ and for all $n\in \mathbb{N}$. 
Therefore, $v_{n}(x)\rightarrow 0$ as $|x|\rightarrow \infty$ uniformly with respect to $n\in \mathbb{N}$.
\end{proof}

\noindent
Now, we study the concentration of maximum points. Let $u_{\e_{n}}$ be a solution to \eqref{Pe} and $v_{n}=|u_{\e_{n}}|(\cdot+\tilde{y}_{n})$, where $(\tilde{y}_{n})$ is given by Proposition \ref{prop4.1}.
Firstly, we note that
\begin{equation}\label{delta}
\|v_{n}\|_{L^{\infty}(\R^{3})}\geq \delta \mbox{ for some } \delta>0, \forall n\in \mathbb{N}.
\end{equation}
Assume by contradiction that $\|v_{n}\|_{L^{\infty}(\R^{3})}\rightarrow 0$. Then, in view of $(f_1)$, there exists $n_{0}\in \mathbb{N}$ such that
$$
\frac{f(\|v_{n}\|^{2}_{L^{\infty}(\R^{3})})}{\|v_{n}\|^{2}_{L^{\infty}(\R^{3})}}<V_{0} \mbox{ and } \|v_{n}\|^{2}_{L^{\infty}(\R^{3})}<\frac{1}{2}.
$$
Using $\langle J'_{\e_{n}}(u_{\e_{n}}),u_{\e_{n}}\rangle=0$ and Lemma \ref{DI} we get
\begin{align*}
[v_{n}]+V_{0}|v_{n}|_{2}^{2}+\int_{\R^{3}}\phi^{t}_{|v_{n}|}|v_{n}|^{2}&\leq \int_{\R^{3}} \frac{f(|v_{n}|^{2})}{|v_{n}|^{2}}|v_{n}|^{4}dx \\
&\leq \int_{\R^{3}} \frac{f(\|v_{n}\|^{2}_{L^{\infty}(\R^{3})})}{\|v_{n}\|^{2}_{L^{\infty}(\R^{3})}} |v_{n}|^{4}dx \\
&\leq V_{0}\int_{\R^{3}} |v_{n}|^{4}dx \\
&\leq V_{0}  \|v_{n}\|^{2}_{L^{\infty}(\R^{3})} \int_{\R^{3}} |v_{n}|^{2}dx \\
&\leq \frac{V_{0}}{2} \int_{\R^{3}} |v_{n}|^{2}dx 
\end{align*}
which implies that $\|v_{n}\|_{V_{0}}\rightarrow 0$ and this is a contradiction because $\|v_{n}\|_{V_{0}}\rightarrow \|v\|_{V_{0}}\neq 0$.

Now, let $p_{n}$ be a global maximum point of $v_{n}$. In view of the second statement in Lemma \ref{moser} and \eqref{delta}, we can see that $p_{n}\in B_{R}(0)$ for some $R>0$. Thus $z_{\e_{n}}=p_{n}+\tilde{y}_{n}$ is a global maximum point of $|u_{\e_{n}}|$ and as a consequence $\eta_{\e_{n}}=\e_{n}z_{\e_{n}}$ is the maximum point of $\hat{u}_{n}=u_{\e_{n}}(x/\e_{n})$ which is a solution to \eqref{P}. Therefore, $\eta_{\e_{n}}=\e_{n}p_{n}+y_{n}\rightarrow y\in M$ and from he continuity of $V$ it follows that $V(\eta_{\e_{n}})\rightarrow V(y)=V_{0}$ as $n\rightarrow \infty$.

In what follows, we prove the power decay estimate of $|\hat{u}_{n}|$. Using  \cite[Lemma $4.3$]{FQT}, we know that there exists a function $w$ such that 
\begin{align}\label{HZ1}
0<w(x)\leq \frac{C}{1+|x|^{3+2s}},
\end{align}
and
\begin{align}\label{HZ2}
(-\Delta)^{s} w+\frac{V_{0}}{2}w\geq 0 \mbox{ in } \R^{3}\setminus B_{R_{1}}(0) 
\end{align}
for some suitable $R_{1}>0$. \\
Since $v_{n}(x)\rightarrow 0$ as $|x|\rightarrow \infty$ uniformly in $n\in \mathbb{N}$ (see Lemma \ref{moser}), there exists $R_{2}>0$ such that
\begin{equation}\label{hzero}
h_{n}=f(v_{n}^{2})v_{n}\leq \frac{V_{0}}{2}v_{n}  \mbox{ in } \R^{3}\setminus B_{R_{2}}(0).
\end{equation}
Let $w_{n}$ be the unique solution to 
$$
(-\Delta)^{s}w_{n}+V_{0}w_{n}=h_{n} \mbox{ in } \R^{3}.
$$
Then $w_{n}(x)\rightarrow 0$ as $|x|\rightarrow \infty$ uniformly in $n\in \mathbb{N}$, and by comparison $0\leq v_{n}\leq w_{n}$ in $\R^{3}$. From \eqref{hzero} we deduce that
\begin{align*}
(-\Delta)^{s}w_{n}+\frac{V_{0}}{2}w_{n}=h_{n}-\frac{V_{0}}{2}w_{n}\leq 0 \mbox{ in } \R^{3}\setminus B_{R_{2}}(0).
\end{align*}
Put $R_{3}=\max\{R_{1}, R_{2}\}$ and we consider
\begin{align}\label{HZ4}
a=\inf_{B_{R_{3}}(0)} w>0 \mbox{ and } \tilde{w}_{n}=(b+1)w-a w_{n}.
\end{align}
where $b=\sup_{n\in \mathbb{N}} \|w_{n}\|_{L^{\infty}(\R^{3})}<\infty$. 
Our purpose is to show that
\begin{equation}\label{HZ5}
\tilde{w}_{n}\geq 0 \mbox{ in } \R^{3}.
\end{equation}
We note that
\begin{align}
&\lim_{|x|\rightarrow \infty} \sup_{n\in \mathbb{N}}\tilde{w}_{n}(x)=0,  \label{HZ0N} \\
&\tilde{w}_{n}\geq ba+w-ba>0 \mbox{ in } B_{R_{3}}(0) \label{HZ0},\\
&(-\Delta)^{s} \tilde{w}_{n}+\frac{V_{0}}{2}\tilde{w}_{n}\geq 0 \mbox{ in } \R^{3}\setminus B_{R_{3}}(0) \label{HZ00}.
\end{align}
We argue by contradiction, and assume that there exists a sequence $(\bar{x}_{j, n})\subset \R^{3}$ such that 
\begin{align}\label{HZ6}
\inf_{x\in \R^{3}} \tilde{w}_{n}(x)=\lim_{j\rightarrow \infty} \tilde{w}_{n}(\bar{x}_{j, n})<0. 
\end{align}
In light of (\ref{HZ0N}), we obtain that $(\bar{x}_{j, n})$ is bounded, and, up to subsequence, we may suppose that there exists $\bar{x}_{n}\in \R^{3}$ such that $\bar{x}_{j, n}\rightarrow \bar{x}_{n}$ as $j\rightarrow \infty$. 
Hence (\ref{HZ6}) yields
\begin{align}\label{HZ7}
\inf_{x\in \R^{3}} \tilde{w}_{n}(x)= \tilde{w}_{n}(\bar{x}_{n})<0.
\end{align}
Using the minimality of $\bar{x}_{n}$ and the representation formula for the fractional Laplacian \cite[Lemma 3.2]{DPV}, we get  
\begin{align}\label{HZ8}
(-\Delta)^{s}\tilde{w}_{n}(\bar{x}_{n})=\frac{c_{3,s}}{2} \int_{\R^{3}} \frac{2\tilde{w}_{n}(\bar{x}_{n})-\tilde{w}_{n}(\bar{x}_{n}+\xi)-\tilde{w}_{n}(\bar{x}_{n}-\xi)}{|\xi|^{3+2s}} d\xi\leq 0.
\end{align}
On the other hand,  $\bar{x}_{n}\in \R^{3}\setminus B_{R_{3}}(0)$ by (\ref{HZ0}) and (\ref{HZ6}), and in view of (\ref{HZ7}) and (\ref{HZ8}), we can obtain 
$$
(-\Delta)^{s} \tilde{w}_{n}(\bar{x}_{n})+\frac{V_{0}}{2}\tilde{w}_{n}(\bar{x}_{n})<0,
$$
which contradicts (\ref{HZ00}).
Consequently, (\ref{HZ5}) holds true and using (\ref{HZ1}) and $v_{n}\leq w_{n}$ we deduce that
\begin{align*}
0\leq v_{n}(x)\leq w_{n}(x)\leq \frac{(b+1)}{a}w(x)\leq \frac{\tilde{C}}{1+|x|^{3+2s}} \mbox{ for all } n\in \mathbb{N}, x\in \R^{3},
\end{align*}
for some constant $\tilde{C}>0$. 
Recalling the definitions of $v_{n}$ and $\hat{u}_{n}$ we can see that  for all $x\in \R^{3}$
\begin{align*}
|\hat{u}_{n}(x)|&=|u_{\e_{n}}|\left(\frac{x}{\e_{n}}\right)=v_{n}\left(\frac{x}{\e_{n}}-\tilde{y}_{n}\right) \\
&\leq \frac{\tilde{C}}{1+|\frac{x}{\e_{n}}-\tilde{y}_{n}|^{3+2s}} \\
&=\frac{\tilde{C} \e_{n}^{3+2s}}{\e_{n}^{3+2s}+|x- \e_{n} \tilde{y}_{n}|^{3+2s}} \\
&\leq \frac{\tilde{C} \e_{n}^{3+2s}}{\e_{n}^{3+2s}+|x-\eta_{\e_{n}}|^{3+2s}}.
\end{align*}

{\bf Acknowledgements.} 
The author warmly thanks the anonymous referee for her/his useful and nice comments on the paper.
%The author would like to thank the anonymous referee for her/his useful comments and valuable suggestions which improved and clarified the paper.

\end{document}